\documentclass{article}

\usepackage{arxiv}

\usepackage[utf8]{inputenc} % allow utf-8 input
\usepackage[T1]{fontenc}    % use 8-bit T1 fonts
\usepackage{hyperref}       % hyperlinks
\usepackage{url}            % simple URL typesetting
\usepackage{booktabs}       % professional-quality tables
\usepackage{amsfonts}       % blackboard math symbols
\usepackage{nicefrac}       % compact symbols for 1/2, etc.
\usepackage{microtype}      % microtypography
\usepackage{lipsum}		% Can be removed after putting your text content
\usepackage{amssymb,amscd,amsmath,enumerate,verbatim,calc,amsthm,bigints,adjustbox}
\usepackage[section]{placeins}
\newtheorem{thm}{Theorem}[section]
\newtheorem{cor}[thm]{Corollary}
\newtheorem{prop}[thm]{Proposition}
\newtheorem{defn}[thm]{Definition}
\newtheorem{rem}[thm]{{ Remark}}

\newtheorem{exmp}[thm]{Example}

\title{Group Classification of Second Order Neutral Differential Equations}

\author{
  Jervin Zen Lobo \\
  Department of Mathematics\\
  St. Xavier's College\\
  Mapusa, Goa. India 403507 \\
  \texttt{zenlobo1990@gmail.com} \\
   \And
 Y. S. Valaulikar \\
  Department of Mathematics\\
  Goa University\\
  Taleigao Plateau, Goa, India 403206 \\
  \texttt{ysv@unigoa.ac.in, ysvgoa@gmail.com} \\
}

\begin{document}
\maketitle

\begin{abstract}
In this paper, we  discuss the  method of obtaining symmetries for  second order non-homogeneous neutral differential equations with variable coefficients. We  use  Taylor's theorem for a function of several variables to obtain a Lie type invariance condition and  the determining equations. Further we make a complete group classification of the second order linear neutral differential equation, for which there is no existing literature. As a special case, we present a complete group classification of the corresponding second order linear delay differential equation.
\end{abstract}

% keywords can be removed

\section{Introduction}
\noindent
Functional differential equations are equations containing some functions along with some of their derivatives at different arguments values. These equations are widely used in different fields such as heat transfer problems, signal processing, evolution of species, etc. (\cite{RD,HJ,KYHS}).  In this paper we restrict our attention to a class of functional differential equations called  neutral differential equations. Neutral differential equations are differential equations in which the unknown function and its derivative appear with time delays. More information on the applications and methods of solving neutral differential equations can be found in \cite{AB,IF,SK,SH,VNO}. Our focus, is to obtain the equivalent symmetries and the corresponding generators of the Lie group under which neutral differential equations are invariant.\\

\noindent
Symmetries are transformations that leave an object unchanged or invariant and are very useful in the formation and study of laws of nature. For more details one can see \cite{EN,OF}. The main idea in Lie's theory of symmetry analysis of differential equations relies on the invariance of the latter under a transformation of dependent and independent variables. In this paper, we do symmetry analysis of the second order neutral differential equation
\begin{equation}
x''(t)=f(t,x(t),x'(t),x(t-r),x'(t-r),x''(t-r)) , \label{1.1}
\end{equation}
where $f$ is defined on $I\times D^{5}, \quad D \text{ is an open set in} \quad \mathbb{R}$ and $I$ is any interval in $\mathbb{R}$. $r>0$ is the delay, $x'(t-r)$ and $x''(t-r)$ mean $\dfrac{dx}{dt}(t-r)$ and $\dfrac{d^2x}{dt^2}(t-r)$ respectively. We shall first find a group under which  equation (\ref{1.1}) is invariant. We call this the admitted Lie group by which we mean that one solution curve is carried to another solution curve of the same equation. We then use this group to obtain the desired symmetries. As most neutral (and delay) differential equations cannot be solved explicitly, group analysis and classification of these differential equations is the best way to study their properties which aid in modeling problems arising in different fields of Mathematics, Physics, Engineering and Mechanics.\\

\noindent
In \cite{JTSM1,JTSM2} symmetries of delay differential equations are obtained by defining  a certain operator equivalent to the canonical Lie-B$\ddot{a}$cklund operator. In \cite{PP},  equivalent symmetries of a second order delay differential equation are obtained.  However, in \cite{PP} too, an operator equivalent to the canonical Lie-B$\ddot{a}$cklund operator and suitable other operators are defined. In \cite{BP}  authors exhaustively describe the Lie symmetries of systems of second order linear ordinary differential equations with constant coefficients over both real and complex fields. They  propose an algebraic approach to obtain bounds for the dimensions of the maximal Lie invariance algebras possessed by such systems. Further, such systems are thoroughly provided their group classification in \cite{SM,SVM}, with extensions to linear systems of second order ordinary differential equations with more than two equations.  Higher order symmetries for ordinary differential equations are studied in \cite{HB}. In \cite{LL} author suggests a group method to study  functional differential equations based on a search of symmetries of underdetermined differential equations by methods of classical and modern group analysis, using the principle of factorization. The method therein, encompasses the use of a basis of invariants consisting of universal and differential invariants . Recently, in \cite{ZLYV} an admitted Lie group for first order delay differential equations with constant coefficients is defined and  corresponding generators of the Lie group for this equation are obtained. The approach in \cite{ZLYV} consists of using Lie Backlund operators to obtain the determining equations. More recently in \cite{ZLYV2}, Lie symmetries of first order neutral differential equations have been found using a Lie type invariance condition obtained from Taylor's theorem for a function of several variables.\\

\noindent
We need the following definition of one-parameter groups of transformations:
\begin{defn}\cite{JH}
	Consider transformations given by,  $\bar{t}_{i}=g_{i}(t_{j},\delta), i,j=1,2,\cdots,n.$ where $\delta$ is the parameter and these transformations, depend continuously on $\delta$.\\
	Let for each $i$, $g_{i}$ be a smooth function of the variables $t_{j}$ having a convergent Taylor series in $\delta$.\\
	We say that, this set of transformations form a one-parameter group of transformations, called \emph{Lie Groups}, if:
	\begin{enumerate}
		\item (\emph{Identity}) The value $\delta=0$ characterizes the identity transformation,\\
		$t_{i}=g_{i}(t_{j},0), i,j=1,2,\cdots,n.$
		\item (\emph{Inverse}) The parameter $-\delta$ characterizes the inverse transformation,\\
		$t_{i}=g_{i}(\bar{t}_{j},-\delta), i,j=1,2,\cdots,n.$
		\item (\emph{Closure}) The product of two transformations of the set is again a transformation of the set.\\
		$\bar{t}_{i}=g_{i}(t_{j},\delta), i,j=1,2,\cdots,n.$, and  $\hat{t}_{i}=g_{i}(\bar{t}_{j},\epsilon), i,j=1,2,\cdots,n.$, then  $\hat{t}_{i}=g_{i}(t_{j},\delta+\epsilon), i,j=1,2,\cdots,n.$
	\end{enumerate}
\end{defn}
\begin{rem}
	We shall be using  following notations.\\
	If the Lie group is given by $\bar{t}=f_{1}(t,x;\delta), \bar{x}=f_{2}(t,x;\delta) $ where $f_{1}$ and $f_{2}$ are smooth functions in $t$ and $x$ having a convergent Taylor series in $\delta$, then\\
	$\omega(t,x)=\dfrac{\partial f_{1}(t,x;0)}{\partial \delta}$ and $\varUpsilon(t,x)=\dfrac{\partial f_{2}(t,x;0)}{\partial \delta}.$
	$\omega$ and $\varUpsilon$ are called coefficients of the infinitesimal transformations or simply infinitesimals.
\end{rem}
\begin{rem}
	By an \emph{equivalent Lie group} we mean, a Lie group of transformations of the dependent and independent variables, and their coefficients, which preserve the differential structure. This group allows simplifying the coefficients of the equations. 
\end{rem}

We have the following definitions:
\begin{defn}
	Let $J$ be an interval in $\mathbb{R}$, and let $D$ be an open set in $\mathbb{R}$. Sometimes $J$ will be $[t_{0},\beta)$, and sometimes it will be $(\alpha,\beta)$, where $\alpha \le t_{0} \le \beta$. Let $f$ map $J \times D^{5}\rightarrow \mathbb{R}$. Conveniently, a second order neutral differential equation is expressed as,\\
	\begin{equation}
	x''(t)=f(t,x(t),x(t-r),x'(t),x'(t-r),x''(t-r)) . \label{3.1}
	\end{equation}
	We consider equation (\ref{3.1}) for $t_{0}\leq t \le \beta$ together with the initial function
	\begin{equation}
	x(t)=\theta (t), \text{for} \quad \gamma \leq t \leq t_{0}. \label{3.2}
	\end{equation}
	where $\theta$ is a given initial function mapping $[\gamma, t_{0}]\rightarrow D$, for some $\gamma\in\mathbb{R}, \quad \gamma<t_{0}.$
\end{defn}
\begin{defn}
	By a solution of the neutral differential equations (\ref{3.1}) with (\ref{3.2}), we mean a differentiable function $x:[\gamma, \beta_{1})\rightarrow D$, for some $\gamma\in\mathbb{R}, \quad \gamma<t_{0}$ $\beta_{1}\in (t_{0},\beta]$ such that,
	\begin{enumerate}
		\item $x(t)=\theta (t)$ for $\gamma \leq t \leq t_{0},$
		\item $x(t)$ reduces equation (\ref{3.1}) to an identity on $t_{0}\leq t \le \beta_{1}.$
	\end{enumerate}
	We understand $x'(t_{0})$ and $x''(t_{0})$ to mean the right-hand derivative.
\end{defn}
In this paper, we do group classification of
\begin{equation}
x''(t)+\alpha(t)x'(t)+\beta(t)x'(t-r)+\gamma(t)x(t)+\rho(t)x(t-r)+\kappa(t)x''(t-r)=h(t). \label{2.5}
\end{equation}

\noindent
By following a completely different approach from the existing literature for delay differential equations, we, in this paper, extend the results of obtaining symmetries of ordinary differential equations \cite{DA,GB,JH,NI} to obtain a complete group classification of second order non homogeneous linear neutral differential equations with variable coefficients. We shall use Taylor's theorem for a function of several variables to obtain a Lie type invariance condition for neutral differential equations. Using this, we suitably define an operator, its prolongation and extension and use it to obtain our determining equations. These equations are then split with respect to the independent variables to obtain an over-determined system of partial differential equations, which are then solved to obtain the most general generator of the Lie group and the corresponding equivalent symmetries. It may be noted that our approach does not lead to any magnification of the delay terms in the determining equations as compared to the existing literature. It is seen that in most cases, we do not get an explicit solution due to the arbitrariness of the variable coefficients. As such, we do not get explicit infinitesimal generators. By then choosing particular values of the variable coefficients or restricting our  equation by choosing certain values of the obtained constants (which does not alter the symmetries obtained), we illustrate the infinitesimal generators of the admitted group, which are explicitly obtained, for these special cases.  We then obtain the group classification of this second order neutral differential equation and as a special case obtain a group classification of the corresponding second order delay differential equation. {\it It is noteworthy to point out here that there is no existing literature on the group classification of neutral differential equations.}\\

\noindent
This paper is organised as follows: In the next section we obtain the Lie type invariance condition for equation (\ref{3.1}). In section $3$, symmetries of equation (\ref{2.5}) are obtained. In the subsequent section $4$,  we illustrate the results obtained.

\section{ Lie type invariance condition for Second Order Neutral Differential Equations}

In this section we obtain Lie type invariance condition for second order neutral differential equations. In order to determine this neutral differential equation completely, we need to specify the delay term.Otherwise the problem is not fully determined.

We establish the following Lie type invariance condition:
\begin{thm}
	For the second order neutral differential equation 
	\begin{equation}
	\dfrac{d^{2}x}{dt^{2}}=F(t,x,t-r,x(t-r),x'(t),x'(t-r),x''(t-r)), \label{3.3}
	\end{equation}
	where $F$ be defined on a 7-dimensional space $I\times D\times I-r\times D^{4}, \quad D\quad \text{is an open set in}\quad \mathbb{R},$ $ I \quad \text{is any interval in} \quad \mathbb{R}$ and $I-r=\{y-r:y\in I\},$ then with the notations, $\omega^{r}=\omega(t-r,x(t-r)) \;$ and $\; \varUpsilon^{r}=\varUpsilon(t-r,x(t-r)),$ the Lie type invariance condition is given by
	\begin{multline*}
	\omega F_{t}+\varUpsilon F_{x}+\omega^{r}F_{t-r}+\varUpsilon^{r}F_{x(t-r)}+\varUpsilon_{[t]}F_{x'(t)}+\varUpsilon_{[t]}^{r}F_{x'(t-r)}+\varUpsilon_{[tt]}^{r}F_{x''(t-r)}=\\
	\varUpsilon_{tt}+(2\varUpsilon_{tx}-\omega_{tt})x'+(\varUpsilon_{xx}-2\omega_{tx})x'^{2}
	-\omega_{xx}x'^{3}+(\varUpsilon_{x}-2\omega_{t})x''-3\omega_{x}x'x''.
	\end{multline*}
	where,
	$$\varUpsilon_{[t]}=D_{t}(\varUpsilon)-x'D_{t}(\omega),$$
	$$\varUpsilon_{[tt]}=D_{t}(\varUpsilon_{[t]})-x''D_{t}(\omega), \quad \text{where} \quad D_{t}=\dfrac{\partial}{\partial t}+x'\dfrac{\partial}{\partial x}+x''\dfrac{\partial}{\partial x'}+\cdots,$$
	$$\varUpsilon_{[t]}^{r}=(\varUpsilon^{r})_{t-r}+((\varUpsilon^{r})_{x(t-r)}-(\omega^{r})_{t-r})x'(t-r)-(x'(t-r))^{2}(\omega^{r})_{x(t-r)},$$
	$\varUpsilon_{[tt]}^{r}=(\varUpsilon_{(t-r)(t-r)}^{r}+(2\varUpsilon_{(t-r)x(t-r)}^{r}-\omega_{(t-r)(t-r)}^{r})x'(t-r)+(\varUpsilon_{x(t-r)x(t-r)}^{r}-2\omega_{(t-r)x(t-r)}^{r})x'(t-r)^{2}\\
	-\omega_{x(t-r)x(t-r)}^{r}x'(t-r)^{3}+(\varUpsilon_{x(t-r)}^{r}-2\omega_{t-r}^{r})x''(t-r)-3\omega_{x(t-r)}^{r}x'(t-r)x''(t-r)).$
\end{thm}
\begin{proof}
	Let the neutral differential equation be invariant under the Lie group\\
	$\bar{t}=t+\delta \omega (t,x)+O(\delta^{2}) ,\; \;\bar{x}=x+\delta \varUpsilon (t,x)+O(\delta^{2}),$ where
	$ \omega 
	,$ and $ \varUpsilon $ are as defined above.\\
	We then naturally define,
	$$\overline{t-r}=t-r+\delta \omega (t-r,x(t-r))+O(\delta^{2})$$ and
	$$\overline{x(t-r)}=x(t-r)+\delta \varUpsilon (t-r,x(t-r))+O(\delta^{2})$$
	Then,
	\begin{equation*}
	\begin{split}
	\dfrac{d\bar{x}}{d\bar{t}}
	&=\left[\dfrac{dx}{dt}+(\varUpsilon_{t}+\varUpsilon_{x}x')\delta+O(\delta^{2})\right]\big[1-(\omega_{t}+\omega_{x}x')\delta+O(\delta^{2})\big]\\
	&=\dfrac{dx}{dt}+[\varUpsilon_{t}+(\varUpsilon_{x}-\omega_{t})x'-\omega_{x}x'^{2}]\delta+O(\delta^{2}).
	\end{split}
	\end{equation*}
	With the notation 
	$D_{t}=\dfrac{\partial}{\partial t}+x'\dfrac{\partial}{\partial x},$
	we can write,
	\begin{equation*}
	\begin{split}
	\dfrac{d\bar{x}}{d\bar{t}}&=\dfrac{dx}{dt}+(D_{t}(\varUpsilon)-x'D_{t}(\omega))\delta+O(\delta^{2})\\
	&=\dfrac{dx}{dt}+\varUpsilon_{[t]}\delta+O(\delta^{2}),
	\end{split}
	\end{equation*}
	where $\varUpsilon_{[t]}=D_{t}(\varUpsilon)-x'D_{t}(\omega)=\varUpsilon_{t}+(\varUpsilon_{x}-\omega_{t})x'-\omega_{x}x'^{2}.$\\
	Considering the second-order extended infinitesimals:
	\begin{equation*}
	\begin{split}
	\dfrac{d^{2}\bar{x}}{d\bar{t}^{2}}
	&= \left(\dfrac{d^{2}x}{dt^{2}}+D_{t}(\varUpsilon_{[t]})\delta+O(\delta^{2})\right)(1-\delta D_{t}(\omega)+O(\delta^{2}))\\
	&= \dfrac{d^{2}x}{dt^{2}}+(D_{t}(\varUpsilon_{[t]})-D_{t}(\omega)x'')\delta+O(\delta^{2}).
	\end{split}
	\end{equation*}
	So 
	$\varUpsilon_{[tt]}=D_{t}(\varUpsilon_{[t]})-x''D_{t}(\omega).$
	As $\varUpsilon_{[t]}$ contains $t, x$ and $x'$, we need to extend the definition of $D_{t}$. Hence we have 
	$D_{t}=\dfrac{\partial}{\partial t}+x'\dfrac{\partial}{\partial x}+x''\dfrac{\partial}{\partial x'}+\cdots$\\
	Expanding $\varUpsilon_{[tt]}$, gives,
	\begin{equation*}
	\varUpsilon_{[tt]}=\varUpsilon_{tt}+(2\varUpsilon_{tx}-\omega_{tt})x'+(\varUpsilon_{xx}-2\omega_{tx})x'^{2}
	-\omega_{xx}x'^{3}+(\varUpsilon_{x}-2\omega_{t})x''-3\omega_{x}x'x''.
	\end{equation*}
	With the notations 
	$\omega^{r}=\omega(t-r,x(t-r)) \;$ and $\; \varUpsilon^{r}=\varUpsilon(t-r,x(t-r)) ,$
	it follows that,
	\begin{equation*}
	\begin{split}
	\overline x'(t-r)
	&= x'(t-r)+[(\varUpsilon^{r})_{t-r}+((\varUpsilon^{r})_{x(t-r)}\\
	&-(\omega^{r})_{t-r})x'(t-r)-(x'(t-r))^{2}(\omega^{r})_{x(t-r)}]\delta+O(\delta^{2}).
	\end{split}
	\end{equation*}
	and
	\begin{equation*}
	\begin{split}
	\overline x''(t-r)
	&= x''(t-r)+\big[\varUpsilon_{(t-r)(t-r)}^{r}+(2\varUpsilon_{(t-r)x(t-r)}^{r}-\omega_{(t-r)(t-r)}^{r})x'(t-r)\\
	&+(\varUpsilon_{x(t-r)x(t-r)}^{r}-2\omega_{(t-r)x(t-r)}^{r})x'(t-r)^{2}-\omega_{x(t-r)x(t-r)}^{r}x'(t-r)^{3}\\
	&+(\varUpsilon_{x(t-r)}^{r}-2\omega_{t-r}^{r})x''(t-r)-3\omega_{x(t-r)}^{r}x'(t-r)x''(t-r)\big]\delta+O(\delta^{2}).\\
	\end{split}
	\end{equation*}
	Let $\varUpsilon_{[t]}^{r}=(\varUpsilon^{r})_{t-r}+((\varUpsilon^{r})_{x(t-r)}-(\omega^{r})_{t-r})x'(t-r)-(x'(t-r))^{2}(\omega^{r})_{x(t-r)}$ and\\
	$\varUpsilon_{[tt]}^{r}=(\varUpsilon_{(t-r)(t-r)}^{r}+(2\varUpsilon_{(t-r)x(t-r)}^{r}-\omega_{(t-r)(t-r)}^{r})x'(t-r)+(\varUpsilon_{x(t-r)x(t-r)}^{r}-2\omega_{(t-r)x(t-r)}^{r})x'(t-r)^{2}
	-\omega_{x(t-r)x(t-r)}^{r}x'(t-r)^{3}+(\varUpsilon_{x(t-r)}^{r}-2\omega_{t-r}^{r})x''(t-r)-3\omega_{x(t-r)}^{r}x'(t-r)x''(t-r)).$\\
	For invariance,\\
	$\dfrac{d^{2}\bar{x}}{d\bar{t}^{2}}=F(\bar{t},\bar{x},\overline{t-r},\overline{x(t-r)},\dfrac{d\bar{x}}{d\bar{t}},\dfrac{d\bar{x}}{d\bar{t}}(\overline{t-r}),\dfrac{d^{2}\bar{x}}{d\bar{t}^{2}}(\overline{t-r})).$\\
	This gives,
	\begin{equation*}
	\begin{split}
	\dfrac{d^{2}x}{dt^{2}}+\varUpsilon_{[tt]}\delta+O(\delta^{2})&=F(t+\delta\omega+O(\delta^{2}),x+\delta\varUpsilon+O(\delta^{2}),
	\\
	&\quad t-r+\delta\omega^{r}+O(\delta^{2}),x(t-r)+\delta\varUpsilon^{r}+O(\delta^{2}),\\
	&\quad \dfrac{dx}{dt}+\delta \varUpsilon _{[t]}+O(\delta^{2}),\dfrac{dx}{dt}(t-r)+\varUpsilon_{[t]}^{r}\delta+O(\delta^{2}),\\
	&\quad \dfrac{d^{2}x}{dt^{2}}(t-r)+\varUpsilon_{[tt]}^{r}\delta+O(\delta^{2}))\\
	&=F(t,x,t-r,x(t-r),x'(t),x'(t-r),x''(t-r)) \\
	& \quad + (\omega F_{t}+\varUpsilon F_{x}+\omega^{r}F_{t-r}+\varUpsilon^{r}F_{x(t-r)}+\varUpsilon_{[t]}F_{x'(t)} \\ & \quad + \varUpsilon_{[t]}^{r}F_{x'(t-r)} 
	+\varUpsilon_{[tt]}^{r}F_{x''(t-r)})\delta+O(\delta^{2}).
	\end{split}
	\end{equation*}
	
	Comparing the coefficient of $\delta$, we get
	\begin{multline}
	\omega F_{t}+\varUpsilon F_{x}+\omega^{r}F_{t-r}+\varUpsilon^{r}F_{x(t-r)}+\varUpsilon_{[t]}F_{x'(t)}+\varUpsilon_{[t]}^{r}F_{x'(t-r)}+\varUpsilon_{[tt]}^{r}F_{x''(t-r)}=\\
	\varUpsilon_{tt}+(2\varUpsilon_{tx}-\omega_{tt})x'+(\varUpsilon_{xx}-2\omega_{tx})x'^{2}
	-\omega_{xx}x'^{3}+(\varUpsilon_{x}-2\omega_{t})x''-3\omega_{x}x'x''. \label{3.4}
	\end{multline}
	Equation (\ref{3.4}) is a Lie type invariance condition.
\end{proof}

We define a prolonged operator for equation (\ref{2.5}) as 
$ \;\;\;\zeta=\omega\dfrac{\partial}{\partial t}+\omega^{r}\dfrac{\partial}{\partial (t-r)}+\varUpsilon\dfrac{\partial}{\partial x}+\varUpsilon^{r}\dfrac{\partial}{\partial x(t-r)}.$ \\

We then, naturally define the extended operator, for equation (\ref{2.5}) as:
\begin{multline}
\zeta^{(1)}=\omega\dfrac{\partial}{\partial t}+\omega^{r}\dfrac{\partial}{\partial (t-r)}+\varUpsilon\dfrac{\partial}{\partial x}+\varUpsilon^{r}\dfrac{\partial}{\partial x(t-r)}+\varUpsilon_{[t]}\dfrac{\partial}{\partial x'}+\varUpsilon_{[t]}^{r}\dfrac{\partial}{\partial x'(t-r)}\\ +\varUpsilon_{[tt]}\dfrac{\partial}{\partial x''}+\varUpsilon_{[tt]}^{r}\dfrac{\partial}{\partial x''(t-r)}. \label{3.5}
\end{multline}
Defining, $\Delta=\dfrac{d^{2}x}{dt^{2}}-F(t,x(t),t-r,x(t-r),x'(t),x'(t-r),x''(t-r))=0$, we get,
\begin{equation}
\zeta^{(1)}\Delta=\varUpsilon_{[tt]}-\omega F_{t}-\varUpsilon F_{x}-\omega^{r}F_{t-r}-\varUpsilon^{r}F_{x(t-r)}-\varUpsilon_{[t]}F_{x'(t)}-\varUpsilon_{[t]}^{r}F_{x'(t-r)}-\varUpsilon_{[tt]}^{r}F_{x''(t-r)}. \label{3.6}
\end{equation}
Comparing equation (\ref{3.6}) with equation (\ref{3.4}), we get,\\
$\varUpsilon_{[tt]}=\varUpsilon_{tt}+(2\varUpsilon_{tx}-\omega_{tt})x'+(\varUpsilon_{xx}-2\omega_{tx})x'^{2}
-\omega_{xx}x'^{3}+(\varUpsilon_{x}-2\omega_{t})x''-3\omega_{x}x'x''.$\\
On substituting $x''=F$ into $\zeta^{(1)}\Delta=0$, we get an invariance condition \\ $ \; \zeta^{(1)}\Delta\mid _{\Delta=0}=0$ for the equation (\ref{2.5}), from which we shall obtain  determining equations.

\section{Symmetries of Non-homogeneous Linear Second Order Neutral Differential Equation}
In this section we shall obtain symmetries of the non homogeneous second order neutral differential equation with continuously differentiable variable coefficients given by,
\begin{equation}
x''(t)+a(t)x'(t)+b(t)x'(t-r)+c(t)x(t)+d(t)x(t-r)+k(t)x''(t-r)=h(t). \label{4.1}
\end{equation}
We have the following.
\begin{prop}
	If $x_{1}(t)$ is an arbitrary solution of equation (\ref{4.1}), then by employing the change of variables $\bar{t}=t, \quad \bar{x}=x-x_{1}(t)$, the neutral differential equation (\ref{4.1}) gets transformed into a homogeneous neutral differential equation
	\begin{equation}
	x''(t)+a(t)x'(t)+b(t)x'(t-r)+c(t)x(t)+d(t)x(t-r)+k(t)x''(t-r)=0. \label{4.2}
	\end{equation}\label{prop}
\end{prop}
\begin{proof}
	The proposition easily follows by substituting $t=\bar{t}$ and $x(t)=\bar{x}+x_{1}(\bar{t})$ in (\ref{4.1}), by noting that\\
	$x_{1}''(t)+a(t)x_{1}'(t)+b(t)x_{1}'(t-r)+c(t)x_{1}(t)+d(t)x_{1}(t-r)+k(t)x_{1}''(t-r)=h(t).$
\end{proof}

\begin{prop}
	By employing a suitable transformation, the neutral differential equation with twice differentiable variable coefficients
	\begin{equation}
	x''(t)+a_{1}(t)x'(t)+b_{1}(t)x'(t-r)+c_{1}(t)x(t)+d_{1}(t)x(t-r)+k_{1}(t)x''(t-r)=0,
	\label{new2}
	\end{equation}
	can be reduced to a one in which the first order ordinary derivative term is missing.
\end{prop} 
\begin{proof}
	By employing the transformation, $x=u(t)s(t)$, where $u(t)\ne0$ is some twice differentiable function in $t$, $s(t)=exp(-\int\limits^{t}\dfrac{a_{1}(\xi)d\xi}{2})+s_{0}$, $s_{0}$ is an arbitrary constant, equation (\ref{new2}), can be reduced to
	$u''(t)+b_{2}(t)u'(t-r)+c_{2}(t)u(t)+d_{2}(t)u(t-r)+k_{2}(t)u''(t-r)=0,$ where $b_{2}(t)=\dfrac{b_{1}(t)s(t-r)+2k(t)s'(t-r)}{s(t)},$ $c_{2}(t)=\dfrac{s''(t)+a_{1}(t)s'(t)+c_{1}(t)s(t)}{s(t)},$ $d_{2}(t)=\dfrac{b_{1}(t)s'(t-r)+d_{1}(t)s(t-r)+k_{1}(t)s''(t-r)}{s(t)}$ and $k_{2}(t)=\dfrac{k_{1}(t)}{u(t)}.$
\end{proof}
This is similar to what is done to second-order ordinary differential equations to remove the coefficient of the first derivative term. This change does not alter the group classification of (\ref{4.1}).\\
Now we shall consider equivalent symmetries of
\begin{equation}
x''(t)+b(t)x'(t-r)+c(t)x(t)+d(t)x(t-r)+k(t)x''(t-r)=0.
\label{4.3}
\end{equation}
Let us specify the delay point,
\begin{equation}
t^{r}=g(t)=t-r. \label{4.4}
\end{equation}
Applying operator $\zeta^{(1)}$ defined by equation (\ref{3.5}) to equation (\ref{4.4}), we get,
\begin{equation}
\omega^{r}=\omega. \label{4.6}
\end{equation}
Applying operator $\zeta^{(1)}$ defined by equation (\ref{3.5}) to equation (\ref{4.3}), we get,
\begin{multline}
\varUpsilon_{tt}+(2\varUpsilon_{tx}-\omega_{tt})x'+(\varUpsilon_{xx}-2\omega_{tx})x'^{2}
-\omega_{xx}x'^{3} \\+(\varUpsilon_{x}-2\omega_{t})(-b(t)x'(t-r)-c(t)x-d(t) x(t-r)\\
-k(t)x''(t-r))-3\omega_{x}x'(-b(t)x'(t-r)-c(t)x-d(t) x(t-r)-k(t)x''(t-r))\\=-\big[\omega(b'(t)x'(t-r)+c'(t)x(t)
+d'(t)x(t-r)+k'(t)x''(t-r))\\ +c(t)\varUpsilon+d(t)\varUpsilon^{r}+b(t)(\varUpsilon_{t-r}^{r}+(\varUpsilon_{x(t-r)}^{r}-\omega_{t-r}^{r})x'(t-r)-\omega_{x(t-r)}^{r}x'^{2}(t-r))\\
+k(t)(\varUpsilon_{(t-r)(t-r)}^{r}+(2\varUpsilon_{(t-r)x(t-r)}^{r}-\omega_{(t-r)(t-r)}^{r})x'(t-r)+(\varUpsilon_{x(t-r)x(t-r)}^{r}\\ -2\omega_{(t-r)x(t-r)}^{r})x'^{2}(t-r)
-\omega_{x(t-r)x(t-r)}^{r}x'^{3}(t-r)+(\varUpsilon_{x(t-r)}^{r}-2\omega_{t-r}^{r})x''(t-r)\\ -3\omega_{x(t-r)}^{r}x'(t-r)x''(t-r))\big]. \label{4.7}
\end{multline}

Splitting equation (\ref{4.7}) with respect to $x'^{3}(t-r)$, we get,\\
$k(t)\omega_{x(t-r)x(t-r)}^{r}=0$, which we can be  easily solved to get,
\begin{equation}
\omega(t,x)=\alpha(t)x+\beta(t) \label{4.8}
\end{equation}
Differentiating equation (\ref{4.7}) with respect to $x''(t-r)$, we get,\\
$k(t)(2\omega_{t}-\varUpsilon_{x})+3k(t)\omega_{x}x'=3k(t)\omega_{x(t-r)}^{r}x'(t-r)-(\omega k'(t)+k(t)(\varUpsilon_{x(t-r)}^{r}-2\omega_{t-r}^{r}))$\\
Splitting this equation with respect to $x'(t-r)$, and using the fact that $k(t) \ne 0$ we get,
$ \; \omega_{x}=0.$\\
This with equation (\ref{4.8}) gives,
\begin{equation}
\omega(t,x)=\beta(t). \label{4.9}
\end{equation}
Splitting equation (\ref{4.7}) with $x'^{2}$, we get,
$\varUpsilon_{xx}=0$, which solves to give,
\begin{equation}
\varUpsilon(t,x)=\gamma(t)x+\rho(t) \label{4.10}
\end{equation}
Substituting equations (\ref{4.9}) and (\ref{4.10}) into the determining equation (\ref{4.7}), we get,
\begin{multline}
\gamma''(t)x+\rho''(t)+(2\gamma'(t)-\beta''(t))x'+(\gamma(t)-2\beta'(t))(-b(t)x'(t-r)-c(t)x \\ -d(t)x(t-r)-k(t)x''(t-r))
=-\big[\beta(t)(b'(t)x'(t-r)+c'(t)x+d'(t)x(t-r) \\ + k'(t)x''(t-r)) + c(t)(\gamma(t)x+\rho(t))+d(t)(\gamma(t-r)x(t-r)
+\rho(t-r))\\ +b(t)(\gamma'(t-r)x(t-r)+\rho'(t-r)+(\gamma(t-r)-\beta'(t-r))x'(t-r))\\ + k(t)(\gamma''(t-r)x(t-r)+\rho''(t-r)
+(2\gamma'(t-r)-\beta''(t-r))x'(t-r)\\+(\gamma(t-r)-2\beta'(t-r))x''(t-r))\big]. \label{4.11}
\end{multline}
From (\ref{4.6}), we have,
\begin{equation}
\beta(t)=\beta(t-r) \label{4.12}
\end{equation}
Splitting (\ref{4.11}) with respect to $x(t)$, we get,
\begin{equation}
\gamma''(t)+2\beta'(t)c(t)+\beta(t)c'(t)=0 \label{4.13}
\end{equation}
Splitting (\ref{4.11}) with respect to $x'(t)$, we get,
\begin{equation}
\gamma(t)=\dfrac{1}{2}[\beta'(t)+c_{1}] \label{4.14}
\end{equation}
Using (\ref{4.12}), we get,
\begin{equation}
\gamma(t)=\gamma(t-r) \label{4.15}
\end{equation}
Splitting (\ref{4.11}) with respect to the constant terms, we get,
\begin{equation}
\rho''(t)+b(t)\rho'(t-r)+c(t)\rho(t)+d(t)\rho(t-r)+k(t)\rho''(t-r)=0. \label{4.16}
\end{equation}
That is, $\rho(t)$ satisfies the homogeneous neutral differential equation of second order given by (\ref{4.2}).\\
Splitting (\ref{4.11}) with respect to $x''(t-r)$, and using (\ref{4.12}) and (\ref{4.15}), we get,
\begin{equation}
\beta(t)k'(t)=0 \label{4.17}
\end{equation}

\begin{thm}
	The neutral differential equation given by equation (\ref{4.3}) for which $k(t)\ne\text{constant}$ admits a two dimensional group generated by $$ \zeta_{1}^{*}=x\dfrac{\partial}{\partial x},\quad \zeta_{2}^{*}=\rho(t)\dfrac{\partial}{\partial x}.$$
\end{thm}
\begin{proof}
	Equation (\ref{4.17}), having to be true for an arbitrary $\beta(t)$ and $k(t)$ implies that for a non-constant $k(t)$, we must have, $\beta(t)=0,$ and consequently,\\
	$\omega(t,x)=0$ and $\varUpsilon(t,x)=\dfrac{c_{1}}{2}x+\rho(t).$\\
	The infinitesimal generator of the Lie group is given by,
	\begin{equation}
	\zeta^{*}=\dfrac{c_{1}}{2}x\dfrac{\partial}{\partial x}+\rho(t)\dfrac{\partial}{\partial x} \label{4.18}
	\end{equation}
	where $c_{1}$ is an arbitrary constant and $\rho(t)$ satisfies (\ref{4.2}).
\end{proof}

Having obtained the infinitesimal generator for the case when $k(t)$ is non-constant, we now perform symmetry analysis and a complete group classification of the second order neutral differential equation given by (\ref{4.1}), for which,
\begin{equation}
k(t)=c_{2} \label{4.19}
\end{equation}  
where $c_{2}$ is an arbitrary constant.\\
Splitting (\ref{4.11}) with respect to $x(t-r)$, and using (\ref{4.15}), we get,
\begin{equation}
k(t)\beta'''(t)+2\beta(t)d'(t)+4\beta'(t)d(t)+2b(t)\gamma'(t)=0 \label{4.20}
\end{equation}
Splitting (\ref{4.11}) with respect to $x'(t-r)$, and using (\ref{4.12}) and (\ref{4.15}), we get,
\begin{equation}
b(t)\beta'(t)+\beta(t)b'(t)=0 \label{4.21}
\end{equation}
Equation (\ref{4.21}) can be easily integrated to give,
\begin{equation}
b(t)\beta(t)=c_{3} \label{4.22}
\end{equation}
where $c_{3}$ is an arbitrary constant.\\
Using (\ref{4.9}), we can rewrite equations (\ref{4.10}), (\ref{4.13}), (\ref{4.14}), (\ref{4.20}) and (\ref{4.22}) respectively as,
\begin{equation}
\varUpsilon(t,x)=\left[\dfrac{1}{2}(\omega_{t}+c_{1})\right]x+\rho(t), \label{4.23}
\end{equation}
\begin{equation}
\omega_{ttt}+4c(t)\omega_{t}+2c'(t)\omega=0, \label{4.24}
\end{equation}
\begin{equation}
\gamma(t)=\dfrac{1}{2}(\omega_{t}+c_{1}), \label{4.25}
\end{equation}
\begin{equation}
c_{2}\omega_{ttt}+2d'(t)\omega(t)+4d(t)\omega_{t}+b(t)\omega_{tt}=0 \label{4.26}
\end{equation}
and
\begin{equation}
\omega(t,x)=\dfrac{c_{3}}{b(t)}, \label{4.27}
\end{equation}
where $c_{1}, c_{2}$ and $c_{3}$ are arbitrary constants.\\
Next we shall obtain a complete classification of equation (\ref{4.3}):

\begin{thm}
	The neutral differential equation given by equation (\ref{4.3}) for which $b(t) \ne 0, d(t) \ne 0, k(t)=c_{2}$ admits a three dimensional group generated by $$  \zeta_{1}^{*}=x\dfrac{\partial}{\partial x},\quad \zeta_{2}^{*}=\dfrac{1}{b(t)}\dfrac{\partial}{\partial t}+\dfrac{x}{2}\left(\dfrac{1}{b(t)}\right)'\dfrac{\partial}{\partial x},\quad \zeta_{3}^{*}=\rho(t)\dfrac{\partial}{\partial x}.$$
\end{thm}
\begin{proof}
	If $c_{3}\ne0$, from (\ref{4.27}) we get
	\begin{equation}
	b(t)=\dfrac{c_{3}}{\omega(t,x)}. \label{4.28}
	\end{equation}
	From (\ref{4.23}), we can write,
	\begin{equation}
	\varUpsilon(t,x)=\big[\dfrac{1}{2}(c_{2}(\dfrac{1}{b(t)})'+c_{1})\big]x+\rho(t). \label{4.29}
	\end{equation}
	Using (\ref{4.28}) in (\ref{4.26}), we get,
	\begin{equation}
	c_{2}\omega\omega_{ttt}+2\omega^{2}d'(t)+4\omega\omega_{t}d(t)+c_{3}\omega_{tt}=0. \label{4.30}
	\end{equation}
	Equation (\ref{4.30}) can be easily integrated to give,
	\begin{equation}
	c_{2}\omega\omega_{tt}-\dfrac{c_{0}}{2}\omega_{t}^{2}+2\omega^{2}d(t)+c_{3}\omega_{t}=c_{4}, \label{4.31}
	\end{equation}
	where $c_{4}$ is an arbitrary constant.\\
	Using (\ref{4.27}), we can solve (\ref{4.31}) for $d(t)$ to get,\\
	$d(t)=\dfrac{1}{2}\left[c_{5}b^{2}(t)+b'(t)+c_{2}\left(\dfrac{b''(t)}{b(t)}-2\left(\dfrac{b'(t)}{b(t)}\right)^{2}+\dfrac{b'(t)}{b^{2}(t)}\right)\right]$, where $c_{5}=c_{2}c_{3}^{2}.$\\
	Since $\omega=\omega^{r},$ we get, $b(t)=b(t-r)$.\\
	Using (\ref{4.27}) in (\ref{4.24}), we get,
	\begin{equation}
	c'(t)-2\dfrac{b'(t)}{b(t)}c(t)=-\dfrac{1}{2}b(t)\left[6\dfrac{b'(t)b''(t)}{b^{3}(t)}-\dfrac{b'''(t)}{b^{2}(t)}-6\dfrac{b'^{3}(t)}{b^{4}(t)}\right]. \label{4.32}
	\end{equation}
	Equation (\ref{4.32}) is a first order linear ordinary differential equation which can be solved to give
	\begin{equation}
	c(t)=\dfrac{1}{2}\left[\dfrac{b''(t)}{b(t)}-\dfrac{3}{2}\left(\dfrac{b'(t)}{b(t)}\right)^{2}+\dfrac{c_{6}}{2}b^{2}(t)\right], \label{4.58}
	\end{equation}
	where $c_{6}$ is an arbitrary constant.\\
	In this case, we have obtained the coefficients of the infinitesimal transformation as
	\begin{equation}
	\omega=\dfrac{c_{3}}{b(t)}, \quad \varUpsilon=\dfrac{x}{2}\left[c_{3}\left(\dfrac{1}{b(t)}\right)'+c_{1}\right]+\rho(t).\label{4.33}
	\end{equation}
	The infinitesimal generator in this case is given by
	\begin{equation}
	\zeta^{*}=\dfrac{c_{1}}{2}x\dfrac{\partial}{\partial x}+c_{3}\left(\dfrac{1}{b(t)}\dfrac{\partial}{\partial t}+\dfrac{x}{2}\left(\dfrac{1}{b(t)}\right)'\dfrac{\partial}{\partial x}\right)+\rho(t)\dfrac{\partial}{\partial x}, \label{4.34}
	\end{equation}
	where $\rho(t)$ is an arbitrary solution of equation (\ref{4.3}).\\
	If $c_{3}=0$ then,
	\begin{equation}
	\omega(t,x)=0, \quad \varUpsilon(t,x)=\dfrac{c_{1}}{2}x+\rho(t). \label{4.35}
	\end{equation}
	The infinitesimal generator is given by,
	\begin{equation}
	\zeta^{*}=\left(\dfrac{c_{1}}{2}x+\rho(t)\right)\dfrac{\partial}{\partial x}. \label{4.36}
	\end{equation}
\end{proof}

\begin{thm}
	The neutral differential equation given by equation (\ref{4.3}) for which $b(t) \ne 0, d(t)=0, k(t)=c_{2}$ admits the infinitesimal generator given by $$  \zeta^{*}=\Phi_{1}(t)\dfrac{\partial}{\partial t}+\Psi_{1}(t,x)\dfrac{\partial}{\partial x},$$ where $\Phi_{1}(t)$ solves $\int^{\omega(t)}\dfrac{c_{2}}{E\tan A} \mathrm{d\theta}-t-c_{9}=0$, for $\omega(t)$ and $A$ is a root (or zero) of $\left[B\ln\left(\dfrac{B^{2}(1+\tan^{2} y)}{c_{2}\theta}\right)+D+2c_{3}y\right]$ for $y$, with
	$B=\sqrt{2c_{7}c_{2}^{2}-c_{3}^{2}}$,\quad	$D=c_{8}B$, \quad
	$E=c_{3}+B,$ and $\Psi_{1}(t,x)=\dfrac{1}{2}\big[(\Phi_{1}(t))_{t}+c_{1}\big]x+\rho(t). $
\end{thm}
\begin{proof}
	If $c_{3}\ne0$, then substituting (\ref{4.28}) into (\ref{4.26}), we get
	\begin{equation}
	c_{2}\omega\omega_{ttt}+c_{3}\omega_{tt}=0. \label{4.37}
	\end{equation}
	This is a non-linear third order differential equation, the solution $\omega(t)$ of which is given by,
	\begin{equation}
	\int^{\omega(t)}\dfrac{c_{2}}{E\tan A} \mathrm{d\theta}-t-c_{9}=0, \label{4.38}
	\end{equation}
	where $A$ is a root (or zero) of $\left[B\ln\left(\dfrac{B^{2}(1+\tan^{2} y)}{c_{2}\theta}\right)+D+2c_{3}y\right]$ for $y$, with\\
	$B=\sqrt{2c_{7}c_{2}^{2}-c_{3}^{2}}, \;\;
	D=c_{8}B \;$ and 
	$ \; E=c_{3}+B.$\\
	It is to be noted that the expression in (\ref{4.38}) may be complex valued and we are finding the zeroes for $y.$ In this solution, $c_{7}, c_{8}$ and $c_{9}$ are arbitrary constants. For obtaining the corresponding infinitesimal generator, we have to solve (\ref{4.38}) for $\omega(t)$.
	The infinitesimal generator in this case is given by
	\begin{equation}
	\zeta^{*}=\Phi_{1}(t)\dfrac{\partial}{\partial t}+\Psi_{1}(t,x)\dfrac{\partial}{\partial x}, \label{4.39}
	\end{equation}
	where $\Phi_{1}(t)$ solves (\ref{4.38}) for $\omega(t)$ and $\Psi_{1}(t,x)=\dfrac{1}{2}\big[(\Phi_{1}(t))_{t}+c_{1}\big]x+\rho(t). $
\end{proof}

\begin{rem}
	In the above, we see that $\zeta^{*}$ is not easy to solve in general. So choosing $B=0$ that is $k(t)=\dfrac{c_{3}}{\sqrt{2c_{7}}}$, we see that, $\omega(t,x)=\dfrac{c_{3}}{c_{2}}(t+c_{10})$, solves (\ref{4.38}).\\
	But the condition $\omega=\omega^{r}$ gives $c_{3}=0.$\\
	Consequently, $\omega(t,x)=0 \quad \varUpsilon(t,x)=\dfrac{1}{2}c_{1}x+\rho(t),$ and the infinitesimal generator is given by
	\begin{equation}
	\zeta^{*}=\dfrac{1}{2}x\dfrac{\partial}{\partial x}+\rho(t)\dfrac{\partial}{\partial x}. \label{4.57}
	\end{equation}
	By considering a very special case in which $c_{2}=1=c_{3}$, we obtain from (\ref{4.26}),
	\begin{equation}
	\omega\omega_{ttt}+\omega_{tt}=0. \label{4.40}
	\end{equation}
	Equation (\ref{4.40}) yields a solution for which some infinitesimal generators can be explicitly found. It's solution is given by
	\begin{equation}
	\int ^{\omega \left( t \right) }\dfrac{\mathrm{d\theta}}{1+c_{11}\tan G}-t-c_{13}=0, \label{4.41}
	\end{equation}
	where $G$ is a root (or zero) of $\left[\ln\left(\dfrac{c_{11}^{2}}{\cos^{2}y}\right)c_{11}-c_{11}\ln\theta+c_{11}c_{12}+2y\right] $ for $y.$\\
	In (\ref{4.41}), $c_{11}, c_{12}$ and $c_{13}$ are arbitrary constants.\\
	The infinitesimal generator in this case is 
	\begin{equation}
	\zeta^{*}=\Phi_{2}(t)\dfrac{\partial}{\partial t}+\Psi_{2}(t,x)\dfrac{\partial}{\partial x}, \label{4.42}
	\end{equation}
	where $\Phi_{2}(t)$ solves (\ref{4.41}) for $\omega(t)$ and $\Psi_{2}(t,x)=\dfrac{1}{2}\big[(\Phi_{2}(t))_{t}+c_{1}\big]x+\rho(t). $
\end{rem}

\begin{cor}
	The neutral differential equation given by equation (\ref{4.3}) for which $b(t) \ne 0, d(t)=0, k(t)=1=c_{3}, c_{11}=0$ admits the three dimensional group given by $$\zeta_{1}^{*}=\dfrac{\partial}{\partial t}, \quad \zeta_{2}^{*}=\dfrac{x}{2}\dfrac{\partial}{\partial x} ,\quad \zeta_{3}^{*}=\rho(t)\dfrac{\partial}{\partial x}.$$
\end{cor}
\begin{proof}
	It can be easily seen that the generators corresponding to $c_{11}=0$ can be explicitly obtained. In this case $\omega(t,x)=c_{14}t+c_{15}$ is a solution of (\ref{4.40}), where $c_{14}$ and $c_{15}$ are arbitrary constants.\\
	The condition $\omega=\omega^{r}$ implies $c_{14}=0.\;$
	Hence $\omega(t,x)=c_{15} \; $ and \\$ \quad \varUpsilon(t,x)=\dfrac{c_{1}}{2}x+\rho(t).$
	If $c_{15}\ne 0,$ then infinitesimal generator is given by
	\begin{equation}
	\zeta^{*}=c_{15}\dfrac{\partial}{\partial t}+\dfrac{1}{2}c_{1}x\dfrac{\partial}{\partial x}+\rho(t)\dfrac{\partial}{\partial x}. \label{4.56}
	\end{equation}
	Using (\ref{4.58}), $c(t)=\dfrac{1}{4}\dfrac{c_{6}c_{3}^{2}}{c_{15}^{2}}$.\\
	If $c_{15}=0$, then the infinitesimal generator is given by (\ref{4.57}).
	Finally, if $c_{3}=0$, then infinitesimal generator is given by (\ref{4.36}).
\end{proof}

\begin{thm}
	The neutral differential equation given by equation (\ref{4.3}) for which $b(t)=0, d(t)\ne0, k(t)=c_{2}$ admits the infinitesimal generator given by $$  \zeta^{*}=\Phi_{3}(t)\dfrac{\partial}{\partial t}+\Psi_{3}(t,x)\dfrac{\partial}{\partial x},$$ where $\Phi_{3}(t)$ solves $c_{2}\omega\omega_{tt}-c_{2}\dfrac{\omega_{t}^{2}}{2}+2\omega^{2}(t)d(t)=c_{16},$ for $\omega(t),$ and $\Psi_{3}(t,x)=\dfrac{1}{2}\big[(\Phi_{3}(t))_{t}+c_{1}\big]x+\rho(t). $
\end{thm}
\begin{proof}
	Then from equation (\ref{4.26}),
	\begin{equation}
	c_{2}\omega_{ttt}+2d'(t)\omega(t)+4d(t)\omega_{t}=0. \label{4.43}
	\end{equation} 
	Equation (\ref{4.43}) can be integrated once to obtain,
	\begin{equation}
	c_{2}\omega\omega_{tt}-c_{2}\dfrac{\omega_{t}^{2}}{2}+2\omega^{2}(t)d(t)=c_{16}, \label{4.44}
	\end{equation}
	where $c_{16}$ is an arbitrary constant.\\
	Equation (\ref{4.44}) is extremely difficult to solve for an arbitrary $d(t)$.\\
	If $\omega(t)=\Phi_{3}(t)$ solves equation (\ref{4.44}), then,
	infinitesimal generator in this case is given by
	\begin{equation}
	\zeta^{*}=\Phi_{3}(t)\dfrac{\partial}{\partial t}+\Psi_{3}(t,x)\dfrac{\partial}{\partial x}, \label{4.46}
	\end{equation}
	where $\Phi_{3}(t)$ solves (\ref{4.44}) for $\omega(t)$ and $\Psi_{3}(t,x)=\dfrac{1}{2}\left[(\Phi_{2}(t))_{t}+c_{1}\right]x+\rho(t). $
\end{proof}

\begin{rem}
	As can be seen the infinitesimal generator given by (\ref{4.46}) cannot be explicitly solved due to the arbitrariness of $d(t)$. However, by choosing a few explicit values of $d(t)$, we obtain  corresponding different infinitesimal generator.
\end{rem}

\begin{cor}
	The neutral differential equation given by equation (\ref{4.3}) for which $b(t)=0, d(t)=e^{t}, k(t)=c_{2}$ admits the five dimensional Lie group generated by\\
	$
	\zeta_{1}^{*}=\Bigl(J_{0}\Bigl(2\lambda\Bigr)\Bigr)^{2}\dfrac{\partial}{\partial t}
	+\dfrac{xe^{t}}{\sqrt{k(t)e^{t}}}J_{0}\Bigl(2\lambda\Bigr)J_{1}\Bigl(2\lambda\Bigr)\dfrac{\partial}{\partial x},
	\quad
	\zeta_{2}^{*}=\Bigl(Y_{0}\Bigl(2\lambda\Bigr)\Bigr)^{2}\dfrac{\partial}{\partial t}
	-\dfrac{xe^{t}}{\sqrt{k(t)e^{t}}}Y_{0}\Bigl(2\lambda\Bigr)Y_{1}\Bigl(2\lambda\Bigr)\dfrac{\partial}{\partial x},$\\$
	\zeta_{3}^{*}=J_{0}\Bigl(2\lambda\Bigr)Y_{0}\Bigl(2\lambda\Bigr)\dfrac{\partial}{\partial t}
	-\dfrac{xe^{t}}{\sqrt{k(t)e^{t}}}\Bigl(J_{1}\Bigl(2\lambda\Bigr)Y_{0}\Bigl(2\lambda\Bigr)
	+J_{0}\Bigl(2\lambda\Bigr)Y_{1}\Bigl(2\lambda\Bigr)\Bigr)\dfrac{\partial}{\partial x},$
	\\$	\zeta_{4}^{*}=\dfrac{x}{2}\dfrac{\partial}{\partial x},\quad \zeta_{5}^{*}=\rho(t)\dfrac{\partial}{\partial x},$ where $\lambda=\dfrac{\sqrt{k(t)e^{t}}}{k(t)}.$
\end{cor}
\begin{proof}
	Taking $d(t)=e^{t}$, equation (\ref{4.44}) becomes $c_{2}\omega\omega_{tt}-c_{2}\dfrac{\omega_{t}^{2}}{2}+2e^{t}\omega^{2}(t)=c_{16}$, which can be solved to give
	\begin{multline} 
	\omega \left( t \right) =\dfrac{1}{4}\,{\frac {{{c_{23}}}^{2}(1+2c_{16})}{c_{22}} \left( 
		{{\sl J}_{0}\left(2\,{\frac {\sqrt {c_{2}{e}^{t}}}{ c_{2}
			}}\right)} \right) ^{2}}+c_{22}\, \left( {{\sl Y}_{0}\left(2\,{
			\frac {\sqrt {c_{2}{e}^{t}}}{ c_{2}}}\right)} \right) ^{2
	} \\ 
	+ c_{23}\,{{\sl J}_{0}\left(2\,{\frac {\sqrt {c_{2}{e}^{t}}}{ c_{2}}}\right)}{{\sl Y}_{0}\left(2\,{\frac {\sqrt {c_{2}{e}^
					{t}}}{ c_{2}}}\right)}, \label{4.59}
	\end{multline}
	where $c_{22}$ and $c_{23}$ are arbitrary constants,\\
	From (\ref{4.23}), we get\\
	$\varUpsilon(t,x)=\dfrac{1}{2}\Bigl[\dfrac{-1}{2}\,{\frac {{e}^{t}{c_{23}}^{2}(1+2c_{16})}{c_{22}\,\sqrt {c_{2}{e}^{t}}}{{\sl J}_{0}\Bigl(2\,{
				\frac {\sqrt {c_{2}{e}^{t}}}{ c_{2}}}\Bigr)}{{\sl J}_{1}
			\Bigl(2\,{\frac {\sqrt {c_{2}{e}^{t}}}{ c_{2}}}\Bigr)}}-
	2\,{\frac {c_{22}\,{e}^{t}}{\sqrt {c_{2}{e}^{t}}}{{\sl Y}_{0}\Bigl(2\,
			{\frac {\sqrt {c_{2}{e}^{t}}}{ c_{2}}}\Bigr)}{{\sl Y}_{1
			}\Bigl(2\,{\frac {\sqrt {c_{2}{e}^{t}}}{ c_{2}}}\Bigr)}}
	\\-{\frac {c_{23}\,{e}^{t}}{\sqrt {c_{2}{e}^{t}}}{{\sl J}_{1}\Bigl(2\,{
				\frac {\sqrt {c_{2}{e}^{t}}}{ c_{2}}}\Bigr)}{{\sl Y}_{0}
			\Bigl(2\,{\frac {\sqrt {c_{2}{e}^{t}}}{ c_{2}}}\Bigr)}}-
	{\frac {c_{23}\,{e}^{t}}{\sqrt {c_{2}{e}^{t}}}{{\sl J}_{0}\Bigl(2\,{
				\frac {\sqrt {c_{2}{e}^{t}}}{ c_{2}}}\Bigr)}{{\sl Y}_{1}
			\Bigl(2\,{\frac {\sqrt {c_{2}{e}^{t}}}{ c_{2}}}\Bigr)}}+c_{1}\Bigr]x+\rho(t)$.\\
	Using (\ref{4.24}), we see that,\\
	$$c\left (t\right )=\left[\int_{}^{}\dfrac{q\left(t\right)}{r\left(t\right)}e^{-4\bigint \dfrac{p_{1}\left(t\right)}{p_{2}\left(t\right)}dt}+c_{24}\right]e^{4\bigint \dfrac{s_{1}(t)}{s_{2}(t)}dt},$$
	where,
	\begin{multline*}
	p_{1}(t)=e^{t}\Bigl[(1+2c_{16})c_{23}^{2}J_{0}\left (\dfrac{2\sqrt{c_{2}e^{t}}}{c_{2}}\right )J_{1}\left (\dfrac{2\sqrt{c_{2}e^{t}}}{c_{2}}\right ) \\ +2c_{22}c_{23}\Bigl(J_{0}\left (\dfrac{2\sqrt{c_{2}e^{t}}}{c_{2}}\right )Y_{1}\left (\dfrac{2\sqrt{c_{2}e^{t}}}{c_{2}}\right )\\
	+J_{1}\left (\dfrac{2\sqrt{c_{2}e^{t}}}{c_{2}}\right )Y_{0}\left (\dfrac{2\sqrt{c_{2}e^{t}}}{c_{2}}\right )\Bigr)+4c_{22}^{2}Y_{0}\left (\dfrac{2\sqrt{c_{2}e^{t}}}{c_{2}}\right )Y_{1}\left (\dfrac{2\sqrt{c_{2}e^{t}}}{c_{2}}\right )\Bigr],
	\end{multline*}
	\begin{multline*}
	p_{2}(t)=\sqrt{c_{2}e^{t}}\Bigl[(1+2c_{16})c_{23}^{2}\left(J_{0}\left (\dfrac{2\sqrt{c_{2}e^{t}}}{c_{2}}\right )\right)^{2} \\ +4c_{22}c_{23}J_{0}\left (\dfrac{2\sqrt{c_{2}e^{t}}}{c_{2}}\right )Y_{0}\left (\dfrac{2\sqrt{c_{2}e^{t}}}{c_{2}}\right )
	+4c_{22}^{2}\left(Y_{0}\left (\dfrac{2\sqrt{c_{2}e^{t}}}{c_{2}}\right )\right)^{2}\Bigr],
	\end{multline*}
	\begin{eqnarray*}
		q\left(t\right)& =& (1+2c_{16})c_{23}^{2}e^{t}\sqrt{c_{2}e^{t}}\left(J_{0}\left (\dfrac{2\sqrt{c_{2}e^{t}}}{c_{2}}\right )\right)^{2}\\
		& & +4c_{22}e^{t}\sqrt{ke^{t}}\Bigl(J_{0}\left (\dfrac{2\sqrt{c_{2}e^{t}}}{c_{2}}\right )Y_{0}\left (\dfrac{2\sqrt{c_{2}e^{t}}}{c_{2}}\right )c_{23}+ \left(Y_{0}\left (\dfrac{2\sqrt{c_{2}e^{t}}}{c_{2}}\right )\right)^{2}c_{22}\Bigr)\\
		& & -4(1+2c_{16})e^{2t}c_{23}^{2}J_{0}\left (\dfrac{2\sqrt{c_{2}e^{t}}}{c_{2}}\right )J_{1}\left (\dfrac{2\sqrt{c_{2}e^{t}}}{c_{2}}\right )\\
		& & -8c_{22}c_{23}e^{2t}\left(J_{0}\left (\dfrac{2\sqrt{c_{2}e^{t}}}{c_{2}}\right )Y_{1}\left (\dfrac{2\sqrt{c_{2}e^{t}}}{c_{2}}\right )+J_{1}\left (\dfrac{2\sqrt{c_{2}e^{t}}}{c_{2}}\right )Y_{0}\left (\dfrac{2\sqrt{c_{2}e^{t}}}{c_{2}}\right )\right)\\
		& & -16e^{2t}c_{22}^{2}Y_{0}\left (\dfrac{2\sqrt{c_{2}e^{t}}}{c_{2}}\right )Y_{1}\left (\dfrac{2\sqrt{c_{2}e^{t}}}{c_{2}}\right ),
	\end{eqnarray*}
	\begin{multline*}
	r(t)=c_{2}\sqrt{c_{2}e^{t}}\Bigl[(1+2c_{16})c_{23}^{2}\left(J_{0}\left (\dfrac{2\sqrt{c_{2}e^{t}}}{c_{2}}\right )\right)^{2} \\ + 4c_{22}\Bigl(J_{0}\left (\dfrac{2\sqrt{c_{2}e^{t}}}{c_{2}}\right )Y_{0}\left (\dfrac{2\sqrt{c_{2}e^{t}}}{c_{2}}\right )c_{23}
	+\left(Y_{0}\left (\dfrac{2\sqrt{c_{2}e^{t}}}{c_{2}}\right )\right)^{2}c_{22}\Bigr)\Bigr],
	\end{multline*}
	\begin{multline*}
	s_{1}(t)=c_{2}^{2}e^{3t}\Bigl[(1+2c_{16})c_{23}^{2}J_{0}\left (\dfrac{2\sqrt{c_{2}e^{t}}}{c_{2}}\right )J_{1}\left (\dfrac{2\sqrt{c_{2}e^{t}}}{c_{2}}\right )+2c_{22}c_{23}\Bigl(J_{0}\left (\dfrac{2\sqrt{c_{2}e^{t}}}{c_{2}}\right )Y_{1}\left (\dfrac{2\sqrt{c_{2}e^{t}}}{c_{2}}\right )\\
	+J_{1}\left (\dfrac{2\sqrt{c_{2}e^{t}}}{c_{2}}\right )Y_{0}\left (\dfrac{2\sqrt{c_{2}e^{t}}}{c_{2}}\right )\Bigr)+4c_{22}^{2}Y_{0}\left (\dfrac{2\sqrt{c_{2}e^{t}}}{c_{2}}\right )Y_{1}\left (\dfrac{2\sqrt{c_{2}e^{t}}}{c_{2}}\right )\Bigr],
	\end{multline*}
	and\\
	\begin{multline*}
	s_{2}(t)=(c_{2}e^{t})^{5/2}\Bigl[(1+2c_{16})c_{23}^{2}\left(J_{0}\left (\dfrac{2\sqrt{c_{2}e^{t}}}{c_{2}}\right )\right)^{2}+4c_{22}Y_{0}\left (\dfrac{2\sqrt{c_{2}e^{t}}}{c_{2}}\right )\Bigl(J_{0}\left (\dfrac{2\sqrt{c_{2}e^{t}}}{c_{2}}\right )c_{23}\\
	+Y_{0}\left (\dfrac{2\sqrt{c_{2}e^{t}}}{c_{2}}\right )c_{22}\Bigr)\Bigr],
	\end{multline*}
	where $c_{24}$ is an arbitrary constant.\\
	The infinitesimal generator is given by
	\begin{multline}
	\zeta^{*}=\dfrac{c_{23}^{2}(1+2c_{16})}{4c_{22}}\Bigl[\Bigl(J_{0}\Bigl(2\dfrac{\sqrt{c_{2}e^{t}}}{c_{2}}\Bigr)\Bigr)^{2}\dfrac{\partial}{\partial t}-\dfrac{xe^{t}}{\sqrt{c_{2}e^{t}}}J_{0}\Bigl(2\dfrac{\sqrt{c_{2}e^{t}}}{c_{2}}\Bigr)J_{1}\Bigl(2\dfrac{\sqrt{c_{2}e^{t}}}{c_{2}}\Bigr)\dfrac{\partial}{\partial x}\Bigr]\\
	+c_{22}\Bigl[\Bigl(Y_{0}\Bigl(2\dfrac{\sqrt{c_{2}e^{t}}}{c_{2}}\Bigr)\Bigr)^{2}\dfrac{\partial}{\partial t}-\dfrac{xe^{t}}{\sqrt{c_{2}e^{t}}}Y_{0}\Bigl(2\dfrac{\sqrt{c_{2}e^{t}}}{c_{2}}\Bigr)Y_{1}\Bigl(2\dfrac{\sqrt{c_{2}e^{t}}}{c_{2}}\Bigr)\dfrac{\partial}{\partial x}\Bigr]\\
	+c_{23}\Bigl[J_{0}\Bigl(2\dfrac{\sqrt{c_{2}e^{t}}}{c_{2}}\Bigr)Y_{0}\Bigl(2\dfrac{\sqrt{c_{2}e^{t}}}{c_{2}}\Bigr)\dfrac{\partial}{\partial t}
	-\dfrac{xe^{t}}{\sqrt{c_{2}e^{t}}}\Bigl(J_{1}\Bigl(2\dfrac{\sqrt{c_{2}e^{t}}}{c_{2}}\Bigr)Y_{0}\Bigl(2\dfrac{\sqrt{c_{2}e^{t}}}{c_{2}}\Bigr)\\
	+J_{0}\Bigl(2\dfrac{\sqrt{c_{2}e^{t}}}{c_{2}}\Bigr)Y_{1}\Bigl(2\dfrac{\sqrt{c_{2}e^{t}}}{c_{2}}\Bigr)\Bigr)\dfrac{\partial}{\partial x}\Bigr]+c_{1}\dfrac{x}{2}\dfrac{\partial}{\partial x}+\rho(t)\dfrac{\partial}{\partial x}. \label{4.60}
	\end{multline}
\end{proof}

\begin{cor}
	The neutral differential equation given by equation (\ref{4.3}) for which $b(t)=0, d(t)=\sin t, k(t)=c_{2}$ admits the five dimensional Lie group generated by 
	\\$
	\zeta_{1}^{*}=\Bigl({\it MathieuC} \Bigl(0,-\dfrac{2}{k(t)},
	\dfrac{-\pi}{4}+\dfrac{t}{2}\Bigr)\Bigr)^{2}\dfrac{\partial}{\partial t}
	\\+\dfrac{x}{2}{\it MathieuC} \Bigl(0,-\dfrac{2}{k(t)},\dfrac{-\pi}{4}+\dfrac{t}{2}\Bigr) {\it MathieuCPrime} \Bigl(0,-\dfrac{2}{k(t)},\dfrac{-\pi}{4}+\dfrac{t}{2}\Bigr)\dfrac{\partial}{\partial x},$\\$
	\zeta_{2}^{*}=\Bigl( {\it MathieuS} \Bigl(0,-\dfrac{2}{k(t)},\dfrac{-\pi}{4}
	+\dfrac{t}{2}\Bigr)  \Bigr) ^{2}\dfrac{\partial}{\partial t}
	+\dfrac{x}{2}{\it MathieuS}\Bigl(0,-\dfrac{2}{k(t)},\dfrac{-\pi}{4}+\dfrac{t}{2}\Bigr) 
	{\it MathieuSPrime} \Bigl(0,-\dfrac{2}{k(t)},\dfrac{-\pi}{4}+\dfrac{t}{2}\Bigr)\dfrac{\partial}{\partial x},$\\$
	\zeta_{3}^{*}={\it MathieuC} \Bigl(0,-\dfrac{2}{k(t)},\dfrac{-\pi}{4}+\dfrac{t}{2}\Bigr) {\it MathieuS} \Bigl(0,-\dfrac{2}{k(t)},\dfrac{-\pi}{4}+\dfrac{t}{2}\Bigr)\dfrac{\partial}{\partial t}
	+\dfrac{x}{4}\Bigl({\it MathieuCPrime} \Bigl(0,-\dfrac{2}{k(t)},\dfrac{-\pi}{4}+\dfrac{t}{2}\Bigr) {\it MathieuS} \Bigl(0,-\dfrac{2}{k(t)},\dfrac{-\pi}{4}+\dfrac{t}{2}\Bigr)
	+{\it MathieuC} \Bigl(0,-\dfrac{2}{k(t)},\dfrac{-\pi}{4}+\dfrac{t}{2}\Bigr) {\it MathieuSPrime} \Bigl(0,-\dfrac{2}{k(t)},\dfrac{-\pi}{4}+\dfrac{t}{2}\Bigr)\Bigr)\dfrac{\partial}{\partial x}, \quad
	\zeta_{4}^{*}=\dfrac{x}{2}\dfrac{\partial}{\partial x}, \quad
	\zeta_{5}^{*}=\rho(t)\dfrac{\partial}{\partial x}. $
\end{cor}
\begin{proof}
	Taking $d(t)=\sin t$, equation (\ref{4.44}) becomes \\ $c_{2}\omega\omega_{tt}-c_{2}\dfrac{\omega_{t}^{2}}{2}+2\omega^{2}(t)\sin t=c_{16}$, which can be solved to give,
	\begin{multline}
	\omega \left( t \right) =\dfrac{1}{4}\,{\frac {{c_{26}}^{2}}{c_{25}}(1+8c_{16}) \left( 
		{\it MathieuC} \left(0,-\dfrac{2}{k(t)},\dfrac{-\pi}{4}+\dfrac{t}{2}\right)  \right) ^{2}} \\
	+c_{25}\, \left( {\it MathieuS} \left(0,-\dfrac{2}{k(t)},\dfrac{-\pi}{4}+\dfrac{t}{2}\right)  \right) ^{2}\\
	+c_{26}\,{\it MathieuC} \left(0,-\dfrac{2}{k(t)},\dfrac{-\pi}{4}+\dfrac{t}{2}\right) {\it MathieuS} \left(0,-\dfrac{2}{k(t)},\dfrac{-\pi}{4}+\dfrac{t}{2}\right),  \label{4.61}
	\end{multline}
	where $c_{25}, c_{26}$ are arbitrary constants.
	Using (\ref{4.47}), equation (\ref{4.23}) gives,\\
	\begin{eqnarray*}
		\varUpsilon(t,x)&=&\dfrac{1}{2}\Bigl[\dfrac{1}{4}\,{\frac {{c_{26}}
				^{2}}{c_{25}}(1+8c_{16}){\it MathieuC} \left(0,-\dfrac{2}{k(t)},\dfrac{-\pi}{4}+\dfrac{t}{2}\right) {\it MathieuCPrime} \left(0,-\dfrac{2}{k(t)},\dfrac{-\pi}{4}+\dfrac{t}{2}\right) 
		}\\
		& &+c_{25}\,{\it MathieuS} \left(0,-\dfrac{2}{k(t)},\dfrac{-\pi}{4}+\dfrac{t}{2}\right)
		{\it MathieuSPrime} \left(0,-\dfrac{2}{k(t)},\dfrac{-\pi}{4}+\dfrac{t}{2}\right) \\
		& &+\dfrac{1}{2}\,c_{26}\,{\it MathieuCPrime} \left(0,-\dfrac{2}{k(t)},\dfrac{-\pi}{4}+\dfrac{t}{2}\right) {\it MathieuS} \left(0,-\dfrac{2}{k(t)},\dfrac{-\pi}{4}+\dfrac{t}{2}\right) \\
		& &+\dfrac{1}{2}
		\,c_{26}\,{\it MathieuC} \left(0,-\dfrac{2}{k(t)},\dfrac{-\pi}{4}+\dfrac{t}{2}\right) 
		{\it MathieuSPrime} \left(0,-\dfrac{2}{k(t)},\dfrac{-\pi}{4}+\dfrac{t}{2}\right) +c_{1}\Bigr]x+\rho(t).
	\end{eqnarray*}
	Using (\ref{4.24}), we see that,
	$$c(t)=e^{-2\int r_{1}(t)\mathrm{d}t}\int \dfrac{q_{1}(t)}{q_{2}(t)}e^{2\int r_{1}(t)\mathrm{d}t}\mathrm{d}t+c_{27}e^{-2\int r_{1}(t)\mathrm{d}t}.$$
	where,
	\begin{multline*}
	r_{1}(t)=\Bigl(c_{26}(1+8c_{16})^{2}{\it MathieuC} \Bigl(0,-\dfrac{2}{k(t)},\dfrac{-\pi}{4}+\dfrac{t}{2}\Bigr){\it MathieuCPrime} \Bigl(0,-\dfrac{2}{k(t)},\dfrac{-\pi}{4}+\dfrac{t}{2}\Bigr)\\
	+2c_{25}c_{26}\Bigl({\it MathieuCPrime} \Bigl(0,-\dfrac{2}{k(t)},\dfrac{-\pi}{4}+\dfrac{t}{2}\Bigr){\it MathieuS} \Bigl(0,-\dfrac{2}{k(t)},\dfrac{-\pi}{4}+\dfrac{t}{2}\Bigr)\\
	+{\it MathieuC} \Bigl(0,-\dfrac{2}{k(t)},\dfrac{-\pi}{4}+\dfrac{t}{2}\Bigr){\it MathieuSPrime} \Bigl(0,-\dfrac{2}{k(t)},\dfrac{-\pi}{4}+\dfrac{t}{2}\Bigr)\Bigr)\\
	+4c_{25}^{2}{\it MathieuS} \Bigl(0,-\dfrac{2}{k(t)},\dfrac{-\pi}{4}+\dfrac{t}{2}\Bigr){\it MathieuSPrime} \Bigl(0,-\dfrac{2}{k(t)},\dfrac{-\pi}{4}+\dfrac{t}{2}\Bigr)\Bigr)\Biggm/\Bigl(c_{26}(1+8c_{16})^{2}\\
	\Bigl({\it MathieuC} \Bigl(0,-\dfrac{2}{k(t)},\dfrac{-\pi}{4}+\dfrac{t}{2}\Bigr)\Bigr)^{2}+4c_{25}c_{26}{\it MathieuC} \Bigl(0,-\dfrac{2}{k(t)},\dfrac{-\pi}{4}+\dfrac{t}{2}\Bigr){\it MathieuS} \Bigl(0,-\dfrac{2}{k(t)},\dfrac{-\pi}{4}+\dfrac{t}{2}\Bigr)\\
	+4c_{25}^{2}\Bigl({\it MathieuS} \Bigl(0,-\dfrac{2}{k(t)},\dfrac{-\pi}{4}+\dfrac{t}{2}\Bigr)\Bigr)^{2}\Bigr),
	\end{multline*}
	\begin{multline*}
	q_{1}(t)=2c_{26}(1+8c_{16})^{2}{\it MathieuC} \Bigl(0,-\dfrac{2}{k(t)},\dfrac{-\pi}{4}+\dfrac{t}{2}\Bigr){\it MathieuCPrime} \Bigl(0,-\dfrac{2}{k(t)},\dfrac{-\pi}{4}+\dfrac{t}{2}\Bigr)\sin t\\
	+4c_{25}c_{26}\Bigl({\it MathieuCPrime} \Bigl(0,-\dfrac{2}{k(t)},\dfrac{-\pi}{4}+\dfrac{t}{2}\Bigr){\it MathieuS} \Bigl(0,-\dfrac{2}{k(t)},\dfrac{-\pi}{4}+\dfrac{t}{2}\Bigr)\\
	+{\it MathieuC} \Bigl(0,-\dfrac{2}{k(t)},\dfrac{-\pi}{4}+\dfrac{t}{2}\Bigr){\it MathieuSPrime} \Bigl(0,-\dfrac{2}{k(t)},\dfrac{-\pi}{4}+\dfrac{t}{2}\Bigr)\Bigr)\sin t\\
	+8c_{25}^{2}{\it MathieuS} \Bigl(0,-\dfrac{2}{k(t)},\dfrac{-\pi}{4}+\dfrac{t}{2}\Bigr){\it MathieuSPrime} \Bigl(0,-\dfrac{2}{k(t)},\dfrac{-\pi}{4}+\dfrac{t}{2}\Bigr)\sin t\\
	+c_{26}(1+8c_{16})^{2}\Bigl({\it MathieuC} \Bigl(0,-\dfrac{2}{k(t)},\dfrac{-\pi}{4}+\dfrac{t}{2}\Bigr)\Bigr)^{2}\cos t+4c_{25}c_{26}{\it MathieuC} \Bigl(0,-\dfrac{2}{k(t)},\dfrac{-\pi}{4}+\dfrac{t}{2}\Bigr)\\
	{\it MathieuS} \Bigl(0,-\dfrac{2}{k(t)},\dfrac{-\pi}{4}+\dfrac{t}{2}\Bigr)\cos t+4c_{25}^{2}\Bigl({\it MathieuS} \Bigl(0,-\dfrac{2}{k(t)},\dfrac{-\pi}{4}+\dfrac{t}{2}\Bigr)\Bigr)^{2}\cos t,
	\end{multline*}
	and
	\begin{multline*}
	q_{2}(t)=k(t)\Bigl(c_{26}(1+8c_{16})^{2}
	\Bigl({\it MathieuC} \Bigl(0,-\dfrac{2}{k(t)},\dfrac{-\pi}{4}+\dfrac{t}{2}\Bigr)\Bigr)^{2}+4c_{25}c_{26}\\
	{\it MathieuC} \Bigl(0,-\dfrac{2}{k(t)},\dfrac{-\pi}{4}+\dfrac{t}{2}\Bigr){\it MathieuS} \Bigl(0,-\dfrac{2}{k(t)},\dfrac{-\pi}{4}+\dfrac{t}{2}\Bigr)
	+4c_{25}^{2}\Bigl({\it MathieuS} \Bigl(0,-\dfrac{2}{k(t)},\dfrac{-\pi}{4}+\dfrac{t}{2}\Bigr)\Bigr)^{2}\Bigr),
	\end{multline*}
	where $c_{27}$ is an arbitrary constant.\\
	The infinitesimal generator in this case is explicitly given by,
	\begin{multline}
	\zeta^{*}=\dfrac{1}{4}\dfrac{c_{26}^{2}}{c_{25}}(1+8c_{16})\Bigl[\Bigl({\it MathieuC} \Bigl(0,-\dfrac{2}{k(t)},\dfrac{-\pi}{4}+\dfrac{t}{2}\Bigr)\Bigr)^{2}\dfrac{\partial}{\partial t}\\
	+\dfrac{x}{2}{\it MathieuC} \Bigl(0,-\dfrac{2}{k(t)},\dfrac{-\pi}{4}+\dfrac{t}{2}\Bigr) {\it MathieuCPrime} \Bigl(0,-\dfrac{2}{k(t)},\dfrac{-\pi}{4}+\dfrac{t}{2}\Bigr)\dfrac{\partial}{\partial x}\Bigr]\\
	+c_{25}\Bigl[\Bigl( {\it MathieuS} \Bigl(0,-\dfrac{2}{k(t)},\dfrac{-\pi}{4}
	+\dfrac{t}{2}\Bigr)  \Bigr) ^{2}\dfrac{\partial}{\partial t}
	+\dfrac{x}{2}{\it MathieuS}\Bigl(0,-\dfrac{2}{k(t)},\dfrac{-\pi}{4}+\dfrac{t}{2}\Bigr) \\
	{\it MathieuSPrime} \Bigl(0,-\dfrac{2}{k(t)},\dfrac{-\pi}{4}+\dfrac{t}{2}\Bigr)\dfrac{\partial}{\partial x}\Bigr]
	+c_{26}\Bigl[{\it MathieuC} \Bigl(0,-\dfrac{2}{k(t)},\dfrac{-\pi}{4}+\dfrac{t}{2}\Bigr) \\{\it MathieuS} \Bigl(0,-\dfrac{2}{k(t)},\dfrac{-\pi}{4}+\dfrac{t}{2}\Bigr)\dfrac{\partial}{\partial t}
	+\dfrac{x}{4}\Bigl({\it MathieuCPrime} \Bigl(0,-\dfrac{2}{k(t)},\dfrac{-\pi}{4}+\dfrac{t}{2}\Bigr) \\
	{\it MathieuS} \Bigl(0,-\dfrac{2}{k(t)},\dfrac{-\pi}{4}+\dfrac{t}{2}\Bigr)
	+{\it MathieuC} \Bigl(0,-\dfrac{2}{k(t)},\dfrac{-\pi}{4}+\dfrac{t}{2}\Bigr) \\
	{\it MathieuSPrime} \Bigl(0,-\dfrac{2}{k(t)},\dfrac{-\pi}{4}+\dfrac{t}{2}\Bigr)\Bigr)\dfrac{\partial}{\partial x}\Bigr] + c_{1}\dfrac{x}{2}\dfrac{\partial}{\partial x}+\rho(t)\dfrac{\partial}{\partial x}\label{4.62}
	\end{multline}
\end{proof}

\begin{cor}
	The neutral differential equation given by equation (\ref{4.3}) for which $b(t)=0, d(t)=t^{m}$ where $m$ is any constant, $k(t)=c_{2}$ admits the five dimensional Lie group generated by \\
	$
	\zeta_{1}^{*}=t(J_{\nu}(\mu))^{2}\dfrac{\partial}{\partial t}
	+x\Bigl(\dfrac{1}{2}(J_{\nu}(\mu))^{2}+\dfrac{2}{m+2}J_{\nu}(\mu)
	\Bigl(-J_{\nu+1}(\mu)+\dfrac{J_{\nu}(\mu)}{2\tau}\Bigr)\tau(m/2+1)\Bigr)\dfrac{\partial}{\partial x},$\\$
	\zeta_{2}^{*}=t(Y_{\nu}(\mu))^{2}\dfrac{\partial}{\partial t}
	+x\Bigl(\dfrac{1}{2}(Y_{\nu}(\mu))^{2}+\dfrac{2}{m+2}Y_{\nu}(\mu)
	\Bigl(-Y_{\nu+1}(\mu)+\dfrac{Y_{\nu}(\mu)}{2\tau}\Bigr)\tau(m/2+1)\Bigr)\dfrac{\partial}{\partial x},$\\$
	\zeta_{3}^{*}=tJ_{\nu}(\mu)Y_{\nu}(\mu)\dfrac{\partial}{\partial t}
	+x\Bigl(\dfrac{1}{2}J_{\nu}(\mu)Y_{\nu}(\mu)
	+\dfrac{1}{m+2}\Bigl(\Bigl(-J_{\nu+1}(\mu)
	+\dfrac{J_{\nu}(\mu)}{2\tau}Y_{\nu}(\mu)+J_{\nu}(\mu)\Bigl(-Y_{\nu+1}(\mu)
	+\dfrac{Y_{\nu}(\mu)}{2\tau}\Bigr)\Bigr)\tau(m/2+1)\Bigr)\Bigr)\dfrac{\partial}{\partial x},
	\quad
	\zeta_{4}^{*}=\dfrac{x}{2}\dfrac{\partial}{\partial x}, \quad
	\zeta_{5}^{*}=\rho(t)\dfrac{\partial}{\partial x},$ where $\nu=(m+2)^{-1},\quad \tau=\sqrt{(k(t))^{-1}}t^{m/2+1}, \mu=2\tau\nu.$
\end{cor}
\begin{proof}
	Taking $d(t)=t^{m}$, where $m$ is any constant, equation (\ref{4.44}) becomes \\ $c_{2}\omega\omega_{tt}-c_{2}\dfrac{\omega_{t}^{2}}{2}+2\omega^{2}(t)t^{m}=c_{16}$, which can be solved to give
	\begin{multline}
	\omega \left( t,x \right) =\dfrac{1}{4}\,{\frac {{c_{29}}^{2}t}{c_{28}}(1+2c_{16})
		\Bigl( {{\sl J}_{ ( m+2 ) ^{-1}}\Bigl(2\,{\frac {\sqrt {c_{2}^{-1}}{t}^{m/2+1}}{m+2}}\Bigr)} \Bigr) ^{2}}+{c_{28}}\,t \Bigl( {
		{\sl Y}_{ ( m+2 ) ^{-1}}\Bigl(2\,{\frac {\sqrt {{c_{2}}^{-1}}{t
				}^{m/2+1}}{m+2}}\Bigr)} \Bigr) ^{2}\\
	+c_{29}\,t{{\sl J}_{ ( m
			+2 ) ^{-1}}\Bigl(2\,{\frac {\sqrt {{c_{2}}^{-1}}{t}^{m/2+1}}{m+2}}
		\Bigr)}{{\sl Y}_{ ( m+2 ) ^{-1}}\Bigl(2\,{\frac {\sqrt {{c_{2}
					}^{-1}}{t}^{m/2+1}}{m+2}}\Bigr)},  \label{4.63}
	\end{multline}
	where $c_{28}$ and $c_{29}$ are arbitrary constants.
	From (\ref{4.23}), we get
	\begin{multline*}
	\varUpsilon(t,x)=\dfrac{1}{2}\Bigl[\dfrac{1}{4}\,{\frac {{c_{29}}^{2}}{c_{28}}(1+2c_{16})
		\Bigl( {{\sl J}_{ ( m+2 ) ^{-1}}\Bigl(2\,{\frac {\sqrt {c_{2}^{-1}}{t}^{m/2+1}}{m+2}}\Bigr)} \Bigr) ^{2}}+\dfrac{1+2c_{16}}{c_{28}(m+2)}\Bigl(c_{29}^{2}{{\sl J}_{ ( m+2 ) ^{-1}}\Bigl(2\,{\frac {\sqrt {c_{2}^{-1}}{t}^{m/2+1}}{m+2}}\Bigr)}\\
	\Bigl(-{{\sl J}_{ ( m+2 ) ^{-1}+1}\Bigl(2\,{\frac {\sqrt {c_{2}^{-1}}{t}^{m/2+1}}{m+2}}\Bigr)}+\dfrac{{{\sl J}_{ ( m+2 ) ^{-1}}\Bigl(2\,{\frac {\sqrt {c_{2}^{-1}}{t}^{m/2+1}}{m+2}}\Bigr)}}{2\sqrt{c_{2}^{-1}}t^{m/2+1}}\Bigr)\sqrt{c_{2}^{-1}}t^{m/2+1}\left(m/2+1\right)\Bigr)\\
	+c_{28}\Bigl( {
		{\sl Y}_{ ( m+2 ) ^{-1}}\Bigl(2\,{\frac {\sqrt {{c_{2}}^{-1}}{t
				}^{m/2+1}}{m+2}}\Bigr)} \Bigr) ^{2}+\dfrac{1}{m+2}\Bigl(4c_{28}{
		{\sl Y}_{ ( m+2 ) ^{-1}}\Bigl(2\,{\frac {\sqrt {{c_{2}}^{-1}}{t
				}^{m/2+1}}{m+2}}\Bigr)}\\
	\Bigl(-{
		{\sl Y}_{ ( m+2 ) ^{-1}+1}\Bigl(2\,{\frac {\sqrt {{c_{2}}^{-1}}{t
				}^{m/2+1}}{m+2}}\Bigr)}+\dfrac{{{\sl Y}_{ ( m+2 ) ^{-1}}\Bigl(2\,{\frac {\sqrt {c_{2}^{-1}}{t}^{m/2+1}}{m+2}}\Bigr)}}{2\sqrt{c_{2}^{-1}}t^{m/2+1}}\Bigr)\sqrt{c_{2}^{-1}}t^{m/2+1}\left(m/2+1\right)\Bigr)\\
	+c_{29}{{\sl J}_{ ( m+2 ) ^{-1}}\Bigl(2\,{\frac {\sqrt {c_{2}^{-1}}{t}^{m/2+1}}{m+2}}\Bigr)}{{\sl Y}_{ ( m+2 ) ^{-1}}\Bigl(2\,{\frac {\sqrt {c_{2}^{-1}}{t}^{m/2+1}}{m+2}}\Bigr)}+\dfrac{1}{m+2}\Bigl(2c_{29}\Bigl(-{{\sl J}_{ ( m+2 ) ^{-1}+1}\Bigl(2\,{\frac {\sqrt {c_{2}^{-1}}{t}^{m/2+1}}{m+2}}\Bigr)}\\
	+\dfrac{{{\sl J}_{ ( m+2 ) ^{-1}}\Bigl(2\,{\frac {\sqrt {c_{2}^{-1}}{t}^{m/2+1}}{m+2}}\Bigr)}}{2\sqrt{c_{2}^{-1}}t^{m/2+1}}\Bigr)\sqrt{c_{2}^{-1}}t^{m/2+1}\left(m/2+1\right){{\sl Y}_{ ( m+2 ) ^{-1}}\Bigl(2\,{\frac {\sqrt {c_{2}^{-1}}{t}^{m/2+1}}{m+2}}\Bigr)}\Bigr)\\
	+\dfrac{1}{m+2}\Bigl(2c_{29}{{\sl J}_{ ( m+2 ) ^{-1}}\Bigl(2\,{\frac {\sqrt {c_{2}^{-1}}{t}^{m/2+1}}{m+2}}\Bigr)}\Bigl(-{{\sl Y}_{ ( m+2 ) ^{-1}+1}\Bigl(2\,{\frac {\sqrt {c_{2}^{-1}}{t}^{m/2+1}}{m+2}}\Bigr)}\\
	+\dfrac{{{\sl Y}_{ ( m+2 ) ^{-1}}\Bigl(2\,{\frac {\sqrt {c_{2}^{-1}}{t}^{m/2+1}}{m+2}}\Bigr)}}{2\sqrt{c_{2}^{-1}}t^{m/2+1}}\Bigr)\sqrt{c_{2}^{-1}}t^{m/2+1}\left(m/2+1\right)\Bigr)+c_{1}\Bigr]x+\rho(t).
	\end{multline*}
	Using (\ref{4.24}), we see that,
	$$c(t)=\left[\int_{}^{}-\dfrac{l(t)}{c_{2}j(t)}e^{-4\bigint\dfrac{e(t)}{j(t)}\mathrm{d}t}\mathrm{d}t+c_{30}\right]e^{4\bigint\dfrac{y(t)}{j(t)}}\mathrm{d}t,$$
	where
	\begin{multline*}
	e(t)=\sqrt{c_{2}^{-1}}(1+2c_{16})c_{29}^{2}{\sl J}_{(3+m)(m+2)^{-1}}\Bigl(2\,{\frac {\sqrt {c_{2}^{-1}}{t}^{m/2+1}}{m+2}}\Bigr){\sl J}_{(m+2)^{-1}}\Bigl(2\,{\frac {\sqrt {c_{2}^{-1}}{t}^{m/2+1}}{m+2}}\Bigr)t^{m/2+1}\\
	+2\sqrt{c_{2}^{-1}}c_{28}c_{29}\Bigl({\sl J}_{(3+m)(m+2)^{-1}}\Bigl(2\,{\frac {\sqrt {c_{2}^{-1}}{t}^{m/2+1}}{m+2}}\Bigr){\sl Y}_{(m+2)^{-1}}\Bigl(2\,{\frac {\sqrt {c_{2}^{-1}}{t}^{m/2+1}}{m+2}}\Bigr)\\
	+{\sl J}_{(m+2)^{-1}}\Bigl(2\,{\frac {\sqrt {c_{2}^{-1}}{t}^{m/2+1}}{m+2}}\Bigr){\sl Y}_{(3+m)(m+2)^{-1}}\Bigl(2\,{\frac {\sqrt {c_{2}^{-1}}{t}^{m/2+1}}{m+2}}\Bigr)\Bigr)t^{m/2+1}\\
	+4\sqrt{c_{2}^{-1}}c_{28}^{2}{\sl Y}_{(3+m)(m+2)^{-1}}\Bigl(2\,{\frac {\sqrt {c_{2}^{-1}}{t}^{m/2+1}}{m+2}}\Bigr){\sl Y}_{(m+2)^{-1}}\Bigl(2\,{\frac {\sqrt {c_{2}^{-1}}{t}^{m/2+1}}{m+2}}\Bigr)t^{m/2+1}\\
	-(1+2c_{16})c_{29}^{2}\Bigl({\sl J}_{(m+2)^{-1}}\Bigl(2\,{\frac {\sqrt {c_{2}^{-1}}{t}^{m/2+1}}{m+2}}\Bigr)\Bigr)^{2}-4c_{28}{\sl Y}_{(m+2)^{-1}}\Bigl(2\,{\frac {\sqrt {c_{2}^{-1}}{t}^{m/2+1}}{m+2}}\Bigr)\\
	\Bigl({\sl J}_{(m+2)^{-1}}\Bigl(2\,{\frac {\sqrt {c_{2}^{-1}}{t}^{m/2+1}}{m+2}}\Bigr)c_{29}+{\sl Y}_{(m+2)^{-1}}\Bigl(2\,{\frac {\sqrt {c_{2}^{-1}}{t}^{m/2+1}}{m+2}}\Bigr)c_{28}\Bigr),
	\end{multline*}
	\begin{multline*}
	j(t)=t\Bigl[(1+2c_{16})c_{29}^{2}\Bigl({\sl J}_{(m+2)^{-1}}\Bigl(2\,{\frac {\sqrt {c_{2}^{-1}}{t}^{m/2+1}}{m+2}}\Bigr)\Bigr)^{2}+4c_{28}{\sl Y}_{(m+2)^{-1}}\Bigl(2\,{\frac {\sqrt {c_{2}^{-1}}{t}^{m/2+1}}{m+2}}\Bigr)\\
	\Bigl({\sl J}_{(m+2)^{-1}}\Bigl(2\,{\frac {\sqrt {c_{2}^{-1}}{t}^{m/2+1}}{m+2}}\Bigr)c_{29}+{\sl Y}_{(m+2)^{-1}}\Bigl(2\,{\frac {\sqrt {c_{2}^{-1}}{t}^{m/2+1}}{m+2}}\Bigr)c_{28}\Bigr)\Bigr],
	\end{multline*}
	\begin{multline*}
	l(t)=4\sqrt{c_{2}^{-1}}(1+2c_{16})c_{29}^{2}{\sl J}_{(3+m)(m+2)^{-1}}\Bigl(2\,{\frac {\sqrt {c_{2}^{-1}}{t}^{m/2+1}}{m+2}}\Bigr){\sl J}_{(m+2)^{-1}}\Bigl(2\,{\frac {\sqrt {c_{2}^{-1}}{t}^{m/2+1}}{m+2}}\Bigr)t^{3m/2+1}\\
	-2mc_{29}^{2}c_{16}\Bigl({\sl J}_{(m+2)^{-1}}\Bigl(2\,{\frac {\sqrt {c_{2}^{-1}}{t}^{m/2+1}}{m+2}}\Bigr)\Bigr)^{2}t^{m}+8\sqrt{c_{2}^{-1}}c_{28}c_{29}\Bigl({\sl J}_{(3+m)(m+2)^{-1}}\Bigl(2\,{\frac {\sqrt {c_{2}^{-1}}{t}^{m/2+1}}{m+2}}\Bigr)\\
	{\sl Y}_{(m+2)^{-1}}\Bigl(2\,{\frac {\sqrt {c_{2}^{-1}}{t}^{m/2+1}}{m+2}}\Bigr)
	+{\sl J}_{(m+2)^{-1}}\Bigl(2\,{\frac {\sqrt {c_{2}^{-1}}{t}^{m/2+1}}{m+2}}\Bigr){\sl Y}_{(3+m)(m+2)^{-1}}\Bigl(2\,{\frac {\sqrt {c_{2}^{-1}}{t}^{m/2+1}}{m+2}}\Bigr)\Bigr)t^{3m/2+1}\\
	+16\sqrt{c_{2}^{-1}}c_{28}^{2}{\sl Y}_{(3+m)(m+2)^{-1}}\Bigl(2\,{\frac {\sqrt {c_{2}^{-1}}{t}^{m/2+1}}{m+2}}\Bigr){\sl Y}_{(m+2)^{-1}}\Bigl(2\,{\frac {\sqrt {c_{2}^{-1}}{t}^{m/2+1}}{m+2}}\Bigr)t^{3m/2+1}\\
	-8c_{29}^{2}c_{16}\Bigl({\sl J}_{(m+2)^{-1}}\Bigl(2\,{\frac {\sqrt {c_{2}^{-1}}{t}^{m/2+1}}{m+2}}\Bigr)\Bigr)^{2}t^{m}-mc_{29}^{2}\Bigl({\sl J}_{(m+2)^{-1}}\Bigl(2\,{\frac {\sqrt {c_{2}^{-1}}{t}^{m/2+1}}{m+2}}\Bigr)\Bigr)^{2}t^{m}\\
	-4mc_{28}c_{29}{\sl J}_{(m+2)^{-1}}\Bigl(2\,{\frac {\sqrt {c_{2}^{-1}}{t}^{m/2+1}}{m+2}}\Bigr){\sl Y}_{(m+2)^{-1}}\Bigl(2\,{\frac {\sqrt {c_{2}^{-1}}{t}^{m/2+1}}{m+2}}\Bigr)t^{m}\\
	-4mc_{28}^{2}\Bigl({\sl Y}_{(m+2)^{-1}}\Bigl(2\,{\frac {\sqrt {c_{2}^{-1}}{t}^{m/2+1}}{m+2}}\Bigr)\Bigr)^{2}t^{m}-4c_{29}^{2}\Bigl({\sl J}_{(m+2)^{-1}}\Bigl(2\,{\frac {\sqrt {c_{2}^{-1}}{t}^{m/2+1}}{m+2}}\Bigr)\Bigr)^{2}t^{m}\\
	-16c_{28}{\sl Y}_{(m+2)^{-1}}\Bigl(2\,{\frac {\sqrt {c_{2}^{-1}}{t}^{m/2+1}}{m+2}}\Bigr)
	\Bigl({\sl J}_{(m+2)^{-1}}\Bigl(2\,{\frac {\sqrt {c_{2}^{-1}}{t}^{m/2+1}}{m+2}}\Bigr)c_{29}+{\sl Y}_{(m+2)^{-1}}\Bigl(2\,{\frac {\sqrt {c_{2}^{-1}}{t}^{m/2+1}}{m+2}}\Bigr)c_{28}\Bigr)t^{m},
	\end{multline*}
	and
	\begin{multline*}
	y(t)=\sqrt{c_{2}^{-1}}c_{29}^{2}(1+2c_{16}){\sl J}_{(3+m)(m+2)^{-1}}\Bigl(2\,{\frac {\sqrt {c_{2}^{-1}}{t}^{m/2+1}}{m+2}}\Bigr){\sl J}_{(m+2)^{-1}}\Bigl(2\,{\frac {\sqrt {c_{2}^{-1}}{t}^{m/2+1}}{m+2}}\Bigr)t^{m/2+1}\\
	+2\sqrt{c_{2}^{-1}}c_{28}c_{29}\Bigl({\sl J}_{(3+m)(m+2)^{-1}}\Bigl(2\,{\frac {\sqrt {c_{2}^{-1}}{t}^{m/2+1}}{m+2}}\Bigr){\sl Y}_{(m+2)^{-1}}\Bigl(2\,{\frac {\sqrt {c_{2}^{-1}}{t}^{m/2+1}}{m+2}}\Bigr)+{\sl J}_{(m+2)^{-1}}\Bigl(2\,{\frac {\sqrt {c_{2}^{-1}}{t}^{m/2+1}}{m+2}}\Bigr)\\
	{\sl Y}_{(3+m)(m+2)^{-1}}\Bigl(2\,{\frac {\sqrt {c_{2}^{-1}}{t}^{m/2+1}}{m+2}}\Bigr)\Bigr)t^{m/2+1}+4\sqrt{c_{2}^{-1}}c_{28}^{2}{\sl Y}_{(3+m)(m+2)^{-1}}\Bigl(2\,{\frac {\sqrt {c_{2}^{-1}}{t}^{m/2+1}}{m+2}}\Bigr)\\
	{\sl Y}_{(m+2)^{-1}}\Bigl(2\,{\frac {\sqrt {c_{2}^{-1}}{t}^{m/2+1}}{m+2}}\Bigr)t^{m/2+1}-(1+2c_{16})c_{29}^{2}\Bigl({\sl J}_{(m+2)^{-1}}\Bigl(2\,{\frac {\sqrt {c_{2}^{-1}}{t}^{m/2+1}}{m+2}}\Bigr)\Bigr)^{2}\\
	-4c_{28}{\sl Y}_{(m+2)^{-1}}\Bigl(2\,{\frac {\sqrt {c_{2}^{-1}}{t}^{m/2+1}}{m+2}}\Bigr)
	\Bigl({\sl J}_{(m+2)^{-1}}\Bigl(2\,{\frac {\sqrt {c_{2}^{-1}}{t}^{m/2+1}}{m+2}}\Bigr)c_{29}+{\sl Y}_{(m+2)^{-1}}\Bigl(2\,{\frac {\sqrt {c_{2}^{-1}}{t}^{m/2+1}}{m+2}}\Bigr)c_{28}\Bigr),
	\end{multline*}
	where  $c_{30}$ is an arbitrary constant.\\
	The infinitesimal generator is given by
	\begin{multline}
	\zeta^{*}=\dfrac{1}{4}\dfrac{c_{29}^{2}}{c_{28}}(1+2c_{16})\Bigl[t\Bigl({{\sl J}_{ ( m+2 ) ^{-1}}\Bigl(2\,{\frac {\sqrt {c_{2}^{-1}}{t}^{m/2+1}}{m+2}}\Bigr)} \Bigr) ^{2}\dfrac{\partial}{\partial t}+\Bigl(\dfrac{x}{2}\Bigl({{\sl J}_{ ( m+2 ) ^{-1}}\Bigl(2\,{\frac {\sqrt {c_{2}^{-1}}{t}^{m/2+1}}{m+2}}\Bigr)} \Bigr) ^{2}\\
	+\dfrac{2x}{m+2}{{\sl J}_{ ( m+2 ) ^{-1}}\Bigl(2\,{\frac {\sqrt {c_{2}^{-1}}{t}^{m/2+1}}{m+2}}\Bigr)}
	\Bigl(-{{\sl J}_{ ( m+2 ) ^{-1}+1}\Bigl(2\,{\frac {\sqrt {c_{2}^{-1}}{t}^{m/2+1}}{m+2}}\Bigr)}+\dfrac{{{\sl J}_{ ( m+2 ) ^{-1}}\Bigl(2\,{\frac {\sqrt {c_{2}^{-1}}{t}^{m/2+1}}{m+2}}\Bigr)}}{2\sqrt{c_{2}^{-1}}t^{m/2+1}}\Bigr)\\
	\sqrt{c_{2}^{-1}}t^{m/2+1}\left(m/2+1\right)\Bigr)\dfrac{\partial}{\partial x}\Bigr]
	+c_{28}\Bigl[t\Bigl({\sl Y}_{ ( m+2 ) ^{-1}}\Bigl(2\,{\frac {\sqrt {{c_{2}}^{-1}}{t
			}^{m/2+1}}{m+2}}\Bigr) \Bigr) ^{2}\dfrac{\partial}{\partial t}+\Bigl(\dfrac{x}{2}\Bigl({\sl Y}_{ ( m+2 ) ^{-1}}\Bigl(2\,{\frac {\sqrt {{c_{2}}^{-1}}{t
			}^{m/2+1}}{m+2}}\Bigr) \Bigr) ^{2}\\
	+\dfrac{2x}{m+2}{
		{\sl Y}_{ ( m+2 ) ^{-1}}\Bigl(2\,{\frac {\sqrt {{c_{2}}^{-1}}{t
				}^{m/2+1}}{m+2}}\Bigr)}
	\Bigl(-{
		{\sl Y}_{ ( m+2 ) ^{-1}+1}\Bigl(2\,{\frac {\sqrt {{c_{2}}^{-1}}{t
				}^{m/2+1}}{m+2}}\Bigr)}+\dfrac{{{\sl Y}_{ ( m+2 ) ^{-1}}\Bigl(2\,{\frac {\sqrt {c_{2}^{-1}}{t}^{m/2+1}}{m+2}}\Bigr)}}{2\sqrt{c_{2}^{-1}}t^{m/2+1}}\Bigr)\\
	\sqrt{c_{2}^{-1}}t^{m/2+1}\left(m/2+1\right)\Bigr)\dfrac{\partial}{\partial x}\Bigr]+c_{29}\Bigl[t{{\sl J}_{ ( m
			+2 ) ^{-1}}\Bigl(2\,{\frac {\sqrt {{c_{2}}^{-1}}{t}^{m/2+1}}{m+2}}
		\Bigr)}{{\sl Y}_{ ( m+2 ) ^{-1}}\Bigl(2\,{\frac {\sqrt {{c_{2}
					}^{-1}}{t}^{m/2+1}}{m+2}}\Bigr)}\dfrac{\partial}{\partial t}\\
	+\Bigl(\dfrac{x}{2}{{\sl J}_{ ( m
			+2 ) ^{-1}}\Bigl(2\,{\frac {\sqrt {{c_{2}}^{-1}}{t}^{m/2+1}}{m+2}}
		\Bigr)}{{\sl Y}_{ ( m+2 ) ^{-1}}\Bigl(2\,{\frac {\sqrt {{c_{2}
					}^{-1}}{t}^{m/2+1}}{m+2}}\Bigr)}+\dfrac{x}{m+2}\Bigl(-{{\sl J}_{ ( m+2 ) ^{-1}+1}\Bigl(2\,{\frac {\sqrt {c_{2}^{-1}}{t}^{m/2+1}}{m+2}}\Bigr)}\\
	+\dfrac{{{\sl J}_{ ( m+2 ) ^{-1}}\Bigl(2\,{\frac {\sqrt {c_{2}^{-1}}{t}^{m/2+1}}{m+2}}\Bigr)}}{2\sqrt{c_{2}^{-1}}t^{m/2+1}}\Bigr)\sqrt{c_{2}^{-1}}t^{m/2+1}\left(m/2+1\right){{\sl Y}_{ ( m+2 ) ^{-1}}\Bigl(2\,{\frac {\sqrt {c_{2}^{-1}}{t}^{m/2+1}}{m+2}}\Bigr)}+\dfrac{x}{m+2}\\
	{{\sl J}_{ ( m+2 ) ^{-1}}\Bigl(2\,{\frac {\sqrt {c_{2}^{-1}}{t}^{m/2+1}}{m+2}}\Bigr)}\Bigl(-{{\sl Y}_{ ( m+2 ) ^{-1}+1}\Bigl(2\,{\frac {\sqrt {c_{2}^{-1}}{t}^{m/2+1}}{m+2}}\Bigr)}
	+\dfrac{{{\sl Y}_{ ( m+2 ) ^{-1}}\Bigl(2\,{\frac {\sqrt {c_{2}^{-1}}{t}^{m/2+1}}{m+2}}\Bigr)}}{2\sqrt{c_{2}^{-1}}t^{m/2+1}}\Bigr)\\
	\sqrt{c_{2}^{-1}}t^{m/2+1}\left(m/2+1\right)\Bigr)\dfrac{\partial}{\partial x}\Bigr]+c_{1}\dfrac{x}{2}\dfrac{\partial}{\partial x}+\rho(t)\dfrac{\partial}{\partial x}. \label{4.64}
	\end{multline}
\end{proof}

\vspace{0.5cm}
\begin{rem}In all our cases above, we gave assumed $k(t)\ne 0$, that is $c_{2}\ne 0.$ However, if $c_{2}=0$, then equation (\ref{4.1}) reduces to a second order delay differential equation. As special cases of our group classification, we study the cases for which $c_{2}=0$.
\end{rem}

\begin{thm}
	The delay differential equation given by equation (\ref{4.3}) for which $b(t) \ne 0, d(t) \ne 0, k(t)=0$ admits a three dimensional group generated by $$  \zeta_{1}^{*}=x\dfrac{\partial}{\partial x},\quad\zeta_{2}^{*}=\dfrac{1}{b(t)}\dfrac{\partial}{\partial t}+\dfrac{x}{2}\left(\dfrac{1}{b(t)}\right)'\dfrac{\partial}{\partial x},\quad \zeta_{3}^{*}=\rho(t)\dfrac{\partial}{\partial x}.$$
\end{thm}
\begin{proof}
	Equation (\ref{4.24}) reduces to ,\\
	$\omega_{ttt}=-(2c'(t)\omega+4c(t)\omega_{t}$ or $\omega\omega_{ttt}=-(2c'(t)\omega^{2}+4c(t)\omega\omega_{t}.$\\
	Integrating this, we get,
	\begin{equation}
	\omega\omega_{tt}-\dfrac{\omega_{t}^{2}}{2}+2c(t)\omega^{2}=c_{31} , \label{4.49}
	\end{equation}
	where $c_{31}$ is an arbitrary constant.\\
	If $c_{3}\ne 0$, then from equation (\ref{4.27}),
	\begin{equation}
	\omega=\dfrac{c_{3}}{b(t)}. \label{4.50}
	\end{equation}
	From equation (\ref{4.23}),
	\begin{equation}
	\varUpsilon(t,x)=x\dfrac{1}{2}\left(c_{3}\left(\dfrac{1}{\beta(t)}\right)'+c_{1}\right)+\rho(t). \label{4.51}
	\end{equation}
	From equation (\ref{4.26}), we get,\\
	$d'(t)-2\dfrac{b'(t)}{b(t)}d(t)=\dfrac{1}{2}\left(b''(t)-2\dfrac{(b'(t))^{2}}{b(t)}\right).$\\
	This is a linear differential equation yielding solution\\
	$d(t)=c_{32}b^{2}(t)+\dfrac{b'(t)}{2}$,\\
	where $c_{32}$ is an arbitrary constant.\\
	From equation (\ref{4.49}),\\
	$c(t)=\dfrac{1}{2}\left[{c_{33}b^{2}(t)}-\dfrac{3}{2}\left(\dfrac{b'(t)}{b(t)}\right)^{2}+\dfrac{b''(t)}{b(t)}\right]$, where $c_{33}=\dfrac{c_{31}}{c_{3}^{2}}$ is an arbitrary constant.\\
	Since, $\omega=\omega^{r}$, we get, $b(t)=b(t-r) ,$ \\
	In this case we get coefficients of the infinitesimal transformation as
	\begin{equation}
	\omega(t,x)=\dfrac{c_{3}}{b(t)}, \quad \varUpsilon(t,x)=x\dfrac{1}{2}\left(c_{3}\left(\dfrac{1}{b(t)}\right)'+c_{1}\right)+\rho(t). \label{4.52}
	\end{equation}
	The infinitesimal generator in this case is
	\begin{equation}
	\zeta^{*}=\dfrac{c_{1}}{2}x\dfrac{\partial}{\partial x}+c_{3}\left(\dfrac{1}{b(t)}\dfrac{\partial}{\partial t}+\dfrac{x}{2}\left(\dfrac{1}{b(t)}\right)'\dfrac{\partial}{\partial x}\right)+\rho(t)\dfrac{\partial}{\partial x} , \label{4.53}
	\end{equation}
	where $\rho(t)$ is an arbitrary solution of equation (\ref{4.3}).\\
	If $c_{3}=0$, then the coefficients of the infinitesimal transformation are given by (\ref{4.35}) and the infinitesimal generator is given by (\ref{4.36}).
\end{proof}

\begin{thm}
	The delay differential equation given by equation (\ref{4.3}) for which $b(t) \ne 0, d(t)=0, k(t)=0$ admits a three dimensional group generated by $$ \zeta_{1}^{*}=\dfrac{\partial}{\partial t},\quad \zeta_{2}^{*}=x\dfrac{\partial}{\partial x},\quad \zeta_{3}^{*}=\rho(t)\dfrac{\partial}{\partial x}.$$
\end{thm}
\begin{proof}
	From equation (\ref{4.26}),\\
	$b(t)\omega_{tt}=0$, which can be solved to give,
	$\omega(t,x)=c_{34}t+c_{35}$, where $c_{34}$ and $c_{35}$ are arbitrary constants.\\
	From equation (\ref{4.49}), $c(t)\omega^{2}(t,x)=c_{36}$, where $c_{36}=\dfrac{c_{31}}{2}+\dfrac{c_{34}^{2}}{4}$, is an arbitrary constant.\\
	Further, as $\omega=\omega^{r}$, we get $c_{34}=0$ and hence, $\omega(t,x)=c_{35}.$\\
	If $c_{35} \ne 0$, then\\
	$c(t)=\dfrac{c_{36}}{c_{35}^{2}}$, $b(t)=\dfrac{c_{3}}{c_{35}}.$\\
	The infinitesimal generator in this case is given by
	\begin{equation}
	\zeta^{*}=c_{35}\dfrac{\partial}{\partial t}+(\dfrac{c_{1}}{2}x+\rho(t))\dfrac{\partial}{\partial x}. \label{4.54}
	\end{equation}
	If $c_{35}=0$, then $\omega(t,x)=0$ and $\varUpsilon(t,x)=\dfrac{c_{1}}{2}x+\rho(t).$\\
	The infinitesimal generator in this case is given by(\ref{4.36}).
\end{proof}

\begin{thm}
	The delay differential equation given by equation (\ref{4.3}) for which $b(t)=0, d(t)\ne0, k(t)=0$ admits a four dimensional group generated by $$ \zeta_{1}^{*}=\dfrac{1}{\sqrt{d(t)}}\dfrac{\partial}{\partial t},\quad \zeta_{2}^{*}=\left[\left(-\dfrac{d'(t)}{d^{3/2}(t)}\right)x\right]\dfrac{\partial}{\partial x}, \quad \zeta_{3}^{*}=\dfrac{x}{2}\dfrac{\partial}{\partial x}, \quad\zeta_{4}^{*}=\rho(t)\dfrac{\partial}{\partial x}.$$
\end{thm}
\begin{proof}
	From equation (\ref{4.26}), we get,\\
	$\omega(t,x)=\sqrt{\dfrac{c_{37}}{d(t)}}$, where $c_{37}$ is an arbitrary constant.\\
	Then from equation (\ref{4.23}),
	\begin{equation*}
	\begin{split}
	\varUpsilon&=\left[\dfrac{1}{2}\left(\left(\sqrt{\dfrac{c_{37}}{d(t)}}\right)'+c_{1}\right)\right]x+\rho(t)\\
	&= \left(-\dfrac{\sqrt{c_{37}}}{4}\dfrac{d'(t)}{d^{3/2}(t)}+\dfrac{c_{1}}{2}\right)x+\rho(t).
	\end{split}
	\end{equation*}
	If $c_{37}\ne 0$, then from equation (\ref{4.49}),\\
	$c(t)=\dfrac{1}{2}\left[\dfrac{c_{31}}{c_{37}}d(t)+\dfrac{d''(t)}{2d(t)}-\dfrac{5}{8}\left(\dfrac{d'(t)}{d(t)}\right)^{2}\right]$.\\
	The infinitesimal generator in this case is given by,
	\begin{equation}
	\zeta^{*}=\sqrt{\dfrac{c_{37}}{d(t)}}\dfrac{\partial}{\partial t}+\left[\left(-\dfrac{d'(t)\sqrt{c_{37}}}{4d^{3/2}(t)}+\dfrac{c_{1}}{2}\right)x+\rho(t)\right]\dfrac{\partial}{\partial x}. \label{4.55}
	\end{equation}
	If $c_{37}=0$, then $\omega(t,x)=0, \varUpsilon(t,x)=\dfrac{c_{1}}{2}x+\rho(t)$.\\
	Hence, the infinitesimal generator in this case is given by (\ref{4.36}).
\end{proof}

\section{Some Illustrative Examples}

\begin{exmp}
	Consider the second order neutral differential equation given by
	$x''(t)+x''(t-\pi)=0$. The solution of this differential equation is $x(t)=\sin t.$\\
	Following the procedure given in the previous section, we can show that,\\
	$\omega(t,x)=c_{38}$, a constant, and $\varUpsilon(t,x)=\dfrac{c_{1}}{2}x+\sin t$.\\
	Solving the system,\\
	$\dfrac{d\bar{t}}{d\delta}=\omega(\bar{t},\bar{x})=c_{38}$ and $\dfrac{d\bar{x}}{d\delta}=\varUpsilon(\bar{t},\bar{x})=\dfrac{c_{1}}{2}x+\sin t$, subject to the conditions, $\bar{t}=t$ and $\bar{x}=x$, when $\delta=0$, we get the above neutral differential equation invariant under the Lie group\\
	$\bar{t}=t+c_{38}\delta$, $\bar{x}=\dfrac{2}{c_{1}}\left[e^{c_{1}\delta/2}\left(\dfrac{c_{1}}{2}x+\sin t\right)-\sin(t+c_{38}\delta)\right]$.\\
	The generators of the Lie group (or vector fields of the symmetry algebra) corresponding to this neutral differential equation are given by,\\
	$\zeta^{*}_{1}=\dfrac{\partial}{\partial t}, \zeta^{*}_{2}=x\dfrac{\partial}{\partial x}$ and $\zeta^{*}_{3}=\sin t\dfrac{\partial}{\partial x}.$
\end{exmp}

\vspace{0.5cm}
\begin{exmp}
	Consider the Cauchy problem,\\
	$x'(t)=\int_{-r}^{0}x(s)\mathrm{d}s$.\\
	This is equivalent to the second order delay differential equation given by,\\
	$x''(t)-x(t)+x(t-r)=0$.\\
	Following the procedure in the previous section, from case 6, we get,\\
	$\omega(t,x)=c_{39}$, where$c_{39}=\sqrt{c_{37}}$, is a constant and $\varUpsilon(t,x)=\dfrac{c_{1}}{2}x+\tilde{x}(t)$.\\
	Solving the system,\\
	$\dfrac{d\bar{t}}{d\delta}=\omega(\bar{t},\bar{x})=c_{39}$ and $\dfrac{d\bar{x}}{d\delta}=\varUpsilon(\bar{t},\bar{x})=\dfrac{c_{1}}{2}x+\tilde{x}(t)$, subject to the conditions, $\bar{t}=t$ and $\bar{x}=x$, when $\delta=0$, we get the above neutral differential equation invariant under the Lie group\\
	$\bar{t}=t+c_{39}\delta$, $\bar{x}=\dfrac{2}{c_{1}}\left[e^{c_{1}\delta/2}\left(\dfrac{c_{1}}{2}x+\tilde{x}(t)\right)-\tilde{x}(t+c_{39}\delta)\right]$.\\
	The generators of the Lie group (or vector fields of the symmetry algebra) corresponding to this delay differential equation are given by,\\
	$\zeta^{*}_{1}=\dfrac{\partial}{\partial t}, \zeta^{*}_{2}=x\dfrac{\partial}{\partial x}$ and $\zeta^{*}_{3}=\tilde{x}(t)\dfrac{\partial}{\partial x}.$
\end{exmp}

\section{Conclusion}
\label{concl}
We have obtained the infinitesimal generators of equation (\ref{4.3}) and based on the various cases we can classify the linear second-order neutral differential equation as follows:

\begin{enumerate}
	\item The neutral differential equation (\ref{4.3}) with $b(t) \ne 0, d(t) \ne 0, k(t) = $ a non constant function, admits the infinitesimal generator given by equation (\ref{4.18}).
	\item The neutral differential equation (\ref{4.3}) with $b(t) \ne 0, d(t) \ne 0, k(t) =$ a non-zero constant, admits the infinitesimal generator given by equation (\ref{4.34}).
	\item The neutral differential equation (\ref{4.3}) with $b(t) \ne 0, d(t) = 0, k(t) =$ a non-zero constant, admits the infinitesimal generator given by equation (\ref{4.39}).
	\item The neutral differential equation (\ref{4.3}) with $b(t) \ne 0, d(t) = 0, k(t)=1$, admits the infinitesimal generator given by equation (\ref{4.56}).
	\item The neutral differential equation (\ref{4.3}) with $b(t) = 0, d(t) \ne 0, k(t) =$ a non-zero constant, admits the infinitesimal generator given by equation (\ref{4.46}).
	\item The neutral differential equation (\ref{4.3}) with $b(t) = 0, d(t) = e^{t}, k(t) =$ a non-zero constant, admits the infinitesimal generator given by equation (\ref{4.60}).
	\item The neutral differential equation (\ref{4.3}) with $b(t) = 0, d(t) = \sin t, k(t) =$ a non-zero constant, admits the infinitesimal generator given by equation (\ref{4.62}).
	\item The neutral differential equation (\ref{4.3}) with $b(t) = 0, d(t) = t^{m},k(t) =$ a non-zero constant, admits the infinitesimal generator given by equation (\ref{4.64}).
	\item The neutral differential equation (\ref{4.3}) with $b(t) = 0, d(t) = 1,k(t) =$ a non-zero constant, admits the infinitesimal generator given by equation (\ref{4.48}).
	\item The delay differential equation (\ref{4.3}) with $b(t) \ne 0, d(t) \ne 0, k(t)= 0$, admits the infinitesimal generator given by equation (\ref{4.53}).
	\item The delay differential equation (\ref{4.3}) with $b(t) \ne 0, d(t) = 0, k(t)= 0$, admits the infinitesimal generator given by equation (\ref{4.54}).
	\item The delay differential equation (\ref{4.3}) with $b(t) = 0, d(t) \ne 0, k(t)=0 $, admits the infinitesimal generator given by equation (\ref{4.55}).
\end{enumerate}

The results can be summarized as a table below:\\

\begin{table}
	\caption{Group Classification of the Second Order Neutral Differential Equation}
	\label{tab:1}       
	\begin{adjustbox}{width=1\textwidth}
		{\small
			\begin{tabular}{|l|r|}
				\hline 
				Type of Second order Neutral Differential Equation & Generators \\ 
				\hline 
				\begin{minipage}{10cm}
					$x''(t)+b(t)x'(t-r)+c(t)x(t)+d(t)x(t-r)+k(t)x''(t-r)=0,$ \\ $k(t)\ne constant$ \\ 
				\end{minipage}&\begin{minipage}{5cm}
					$ \zeta_{1}^{*}=x\dfrac{\partial}{\partial x},\\ \zeta_{2}^{*}=\rho(t)\dfrac{\partial}{\partial x}$ 
				\end{minipage} \\ 
				\hline
				\begin{minipage}{10cm}
					$x''(t)+b(t)x'(t-r)+c(t)x(t)+d(t)x(t-r)+k(t)x''(t-r)=0,$ \\$k(t)=c_{2} ,$\\
					$d(t)=\dfrac{1}{2}\left[c_{5}b^{2}(t)+b'(t)+c_{2}\left(\dfrac{b''(t)}{b(t)}-2\left(\dfrac{b'(t)}{b(t)}\right)^{2}+\dfrac{b'(t)}{b^{2}(t)}\right)\right],$ \\ $c(t)=\dfrac{1}{2}\left[\dfrac{b''(t)}{b(t)}-\dfrac{3}{2}\left(\dfrac{b'(t)}{b(t)}\right)^{2}+\dfrac{c_{6}}{2}b^{2}(t)\right] $ 
				\end{minipage}&\begin{minipage}{5cm}
					$ \zeta_{1}^{*}=x\dfrac{\partial}{\partial x},\\\zeta_{2}^{*}=\dfrac{1}{b(t)}\dfrac{\partial}{\partial t}+\dfrac{x}{2}\left(\dfrac{1}{b(t)}\right)'\dfrac{\partial}{\partial x},\\ \zeta_{3}^{*}=\rho(t)\dfrac{\partial}{\partial x}$ 
				\end{minipage} \\ 
				\hline 
				\begin{minipage}{10cm}
					$x''(t)+b(t)x'(t-r)+c(t)x(t)+k(t)x''(t-r)=0,$\\ $c(t)=\dfrac{1}{2}\left[\dfrac{b''(t)}{b(t)}-\dfrac{3}{2}\left(\dfrac{b'(t)}{b(t)}\right)^{2}+\dfrac{c_{6}}{2}b^{2}(t)\right],$\\ $k(t)=\dfrac{c_{3}}{\sqrt{2c_{7}}}$
				\end{minipage}& \begin{minipage}{5cm}
					$\zeta_{1}^{*}=\dfrac{x}{2}\dfrac{\partial}{\partial x},$ \\$\zeta_{2}^{*}=\rho(t)\dfrac{\partial}{\partial x}.$ 
				\end{minipage}\\ 
				\hline
				\begin{minipage}{10cm}
					$x''(t)+b(t)x'(t-r)+c(t)x(t)+k(t)x''(t-r)=0,$\\ $k(t)=1$ \\$c(t)=\dfrac{1}{4}\dfrac{c_{6}c_{3}^{2}}{c_{15}^{2}},$\\$c_{3}=1 $
				\end{minipage}& \begin{minipage}{5cm}
					$\zeta_{1}^{*}=\dfrac{\partial}{\partial t},$ \\$\zeta_{2}^{*}=\dfrac{x}{2}\dfrac{\partial}{\partial x} $,\\$ \zeta_{3}^{*}=\rho(t)\dfrac{\partial}{\partial x}.$ 
				\end{minipage}\\
				\hline 
				\begin{minipage}{8cm}
					$x''(t)+c(t)x(t)+d(t)x(t-r)+k(t)x''(t-r)=0,\\$ $d(t)=1,$
					\begin{multline*}
					c(t)=\Bigl(2c_{18}^{2}\cos\Bigl(\dfrac{4t}{\sqrt{c_{2}}}\Bigr)-2c_{19}^{2}\cos\Bigl(\dfrac{4t}{\sqrt{c_{2}}}\Bigr)-4c_{18}c_{19}\\
					\sin\Bigl(\dfrac{4t}{\sqrt{c_{2}}}\Bigr)
					-4\cos\Bigl(\dfrac{4t}{\sqrt{c_{2}}}\Bigr)\sqrt{c_{17}}c_{19}
					-4\sin\Bigl(\dfrac{4t}{\sqrt{c_{2}}}\Bigr)\sqrt{c_{17}}c_{18}-\\c_{20}c_{2}\Bigr) 
					\Biggm/
					\Bigl(2\cos\Bigl(\dfrac{4t}{\sqrt{c_{2}}}\Bigr)c_{2}c_{18}^{2}-2\cos\Bigl(\dfrac{4t}{\sqrt{c_{2}}}\Bigr)c_{2}c_{19}^{2}\\-4\sin\Bigl(\dfrac{4t}{\sqrt{c_{2}}}\Bigr)c_{2}c_{18}c_{19}
					-4\cos\Bigl(\dfrac{4t}{\sqrt{c_{2}}}\Bigr)c_{2}\sqrt{c_{17}}c_{19}\\
					-4\sin\Bigl(\dfrac{4t}{\sqrt{c_{2}}}\Bigr)c_{2}\sqrt{c_{17}}c_{18}
					-6c_{2}c_{18}^{2}\\-6c_{2}c_{19}^{2}-2c_{2}c_{16}\Bigr)
					\end{multline*}
				\end{minipage}& 
				\begin{minipage}{5cm}
					$\zeta_{1}^{*}=\dfrac{\partial}{\partial t},$\\ $\zeta_{2}^{*}=\dfrac{x}{2}\dfrac{\partial}{\partial x}$ 
					\begin{multline*}
					\zeta_{3}^{*}=\sin\left(\dfrac{2t}{\sqrt{k(t)}}\right)\dfrac{\partial}{\partial t}\\
					+\dfrac{x}{\sqrt{k(t)}}\cos\left(\dfrac{2t}{\sqrt{k(t)}}\right)\dfrac{\partial}{\partial x},
					\end{multline*}
					\begin{multline*}
					\zeta_{4}^{*}=\cos\left(\dfrac{2t}{\sqrt{k(t)}}\right)\dfrac{\partial}{\partial t}\\
					-\dfrac{x}{\sqrt{k(t)}}\sin\left(\dfrac{2t}{\sqrt{k(t)}}\right)\dfrac{\partial}{\partial x},
					\end{multline*}
					$ \zeta_{5}^{*}=\rho(t)\dfrac{\partial}{\partial x}. $
				\end{minipage}\\ 
				
				\hline
		\end{tabular}}
	\end{adjustbox}
\end{table}

\begin{table}
	\caption{Group Classification of the Second Order Neutral Differential Equation}
	\label{tab:2}       
	\begin{adjustbox}{width=1\textwidth}
		{\small
			\begin{tabular}{|l|r|}
				\hline 
				Type of Second order Neutral Differential Equation & Generators \\ 
				\hline 
				\begin{minipage}{10cm}
					$x''(t)+c(t)x(t)+d(t)x(t-r)+k(t)x''(t-r)=0,\\$ $d(t)=e^{t},$\\
					$$c\left (t\right )=\left[\int_{}^{}\dfrac{q\left(t\right)}{r\left(t\right)}e^{-4\bigint \dfrac{p_{1}\left(t\right)}{p_{2}\left(t\right)}dt}+c_{24}\right]e^{4\bigint \dfrac{s_{1}(t)}{s_{2}(t)}dt}$$
					\begin{multline*}
					p_{1}(t)=e^{t}\Bigl[(1+2c_{16})c_{23}^{2}J_{0}(\lambda)J_{1}(\lambda)+2c_{22}c_{23}\Bigl(J_{0}(\lambda)Y_{1}(\lambda)
					\\+J_{1}(\lambda)Y_{0}(\lambda)\Bigr)+4c_{22}^{2}Y_{0}(\lambda)Y_{1}(\lambda)\Bigr]
					\end{multline*}
					\begin{multline*}
					p_{2}(t)=\sqrt{c_{2}e^{t}}\Bigl[(1+2c_{16})c_{23}^{2}\left(J_{0}(\lambda)\right)^{2}\\+4c_{22}c_{23}J_{0}(\lambda)Y_{0}(\lambda)
					+4c_{22}^{2}\left(Y_{0}(\lambda)\right)^{2}\Bigr]
					\end{multline*}
					\begin{multline*}
					q\left(t\right)=(1+2c_{16})c_{23}^{2}e^{t}\sqrt{c_{2}e^{t}}\left(J_{0}(\lambda)\right)^{2}
					+4c_{22}e^{t}\sqrt{ke^{t}}\Bigl(J_{0}(\lambda)Y_{0}(\lambda)\\c_{23}+\left(Y_{0}(\lambda)\right)^{2}c_{22}\Bigr)
					-4(1+2c_{16})e^{2t}c_{23}^{2}J_{0}(\lambda)J_{1}(\lambda)
					-8c_{22}c_{23}\\e^{2t}\left(J_{0}(\lambda)Y_{1}(\lambda)+J_{1}(\lambda)Y_{0}(\lambda)\right)
					-16e^{2t}c_{22}^{2}Y_{0}(\lambda)Y_{1}(\lambda)
					\end{multline*}
					\begin{multline*}
					r(t)=c_{2}\sqrt{c_{2}e^{t}}\Bigl[(1+2c_{16})c_{23}^{2}\left(J_{0}(\lambda)\right)^{2}+4c_{22}\Bigl(J_{0}(\lambda)Y_{0}(\lambda)c_{23}\\
					+\left(Y_{0}(\lambda)\right)^{2}c_{22}\Bigr)\Bigr].
					\end{multline*}
					\begin{multline*}
					s_{1}(t)=c_{2}^{2}e^{3t}\Bigl[(1+2c_{16})c_{23}^{2}J_{0}(\lambda)J_{1}(\lambda)+2c_{22}c_{23}\Bigl(J_{0}(\lambda)Y_{1}(\lambda)\\
					+J_{1}(\lambda)Y_{0}(\lambda)\Bigr)+4c_{22}^{2}Y_{0}(\lambda)Y_{1}(\lambda)\Bigr]
					\end{multline*}
					\begin{multline*}
					s_{2}(t)=(c_{2}e^{t})^{5/2}\Bigl[(1+2c_{16})c_{23}^{2}\left(J_{0}(\lambda)\right)^{2}+4c_{22}Y_{0}(\lambda)\Bigl(J_{0}(\lambda)c_{23}\\
					+Y_{0}(\lambda)c_{22}\Bigr)\Bigr]
					\end{multline*}
				\end{minipage}& 
				\begin{minipage}{5cm}
					With $\lambda=\dfrac{\sqrt{k(t)e^{t}}}{k(t)},$\\
					\begin{multline*}
					\zeta_{1}^{*}=\Bigl(J_{0}\Bigl(2\lambda\Bigr)\Bigr)^{2}\dfrac{\partial}{\partial t}\\
					+\dfrac{xe^{t}}{\sqrt{k(t)e^{t}}}J_{0}\Bigl(2\lambda\Bigr)J_{1}\Bigl(2\lambda\Bigr)\dfrac{\partial}{\partial x}
					\end{multline*}
					\begin{multline*}
					\zeta_{2}^{*}=\Bigl(Y_{0}\Bigl(2\lambda\Bigr)\Bigr)^{2}\dfrac{\partial}{\partial t}\\
					-\dfrac{xe^{t}}{\sqrt{k(t)e^{t}}}Y_{0}\Bigl(2\lambda\Bigr)Y_{1}\Bigl(2\lambda\Bigr)\dfrac{\partial}{\partial x}
					\end{multline*}
					\begin{multline*}
					\zeta_{3}^{*}=J_{0}\Bigl(2\lambda\Bigr)Y_{0}\Bigl(2\lambda\Bigr)\dfrac{\partial}{\partial t}\\
					-\dfrac{xe^{t}}{\sqrt{k(t)e^{t}}}\Bigl(J_{1}\Bigl(2\lambda\Bigr)Y_{0}\Bigl(2\lambda\Bigr)\\
					+J_{0}\Bigl(2\lambda\Bigr)Y_{1}\Bigl(2\lambda\Bigr)\Bigr)\dfrac{\partial}{\partial x}
					\end{multline*}
					$	\zeta_{4}^{*}=\dfrac{x}{2}\dfrac{\partial}{\partial x},$\\
					$ \zeta_{5}^{*}=\rho(t)\dfrac{\partial}{\partial x}. $
				\end{minipage}\\

				\hline
		\end{tabular}}
	\end{adjustbox}
\end{table}

\begin{table}
	\caption{Group Classification of the Second Order Neutral Differential Equation}
	\label{tab:3}
	\begin{adjustbox}{width=1\textwidth}
		{\small
			\begin{tabular}{|l|r|}
				\hline 
				Type of Second order Neutral Differential Equation & Generators \\ 
				\hline 
				\begin{minipage}{10cm}
					$x''(t)+c(t)x(t)+d(t)x(t-r)+k(t)x''(t-r)=0,\\$ $d(t)=\sin t,$
					$$c(t)=e^{-2\int r_{1}(t)\mathrm{d}t}\int \dfrac{q_{1}(t)}{q_{2}(t)}e^{2\int r_{1}(t)\mathrm{d}t}\mathrm{d}t+c_{27}e^{-2\int r_{1}(t)\mathrm{d}t}.$$
					where,
					\begin{multline*}
					r_{1}(t)=\Bigl(c_{26}(1+8c_{16})^{2}{\it MathieuC} \Bigl(0,-\dfrac{2}{k(t)},\dfrac{-\pi}{4}+\dfrac{t}{2}\Bigr){\it MathieuCPrime} \\
					\Bigl(0,-\dfrac{2}{k(t)},\dfrac{-\pi}{4}+\dfrac{t}{2}\Bigr)
					+2c_{25}c_{26}\Bigl({\it MathieuCPrime} \Bigl(0,-\dfrac{2}{k(t)},\dfrac{-\pi}{4}+\dfrac{t}{2}\Bigr)\\
					{\it MathieuS} \Bigl(0,-\dfrac{2}{k(t)},\dfrac{-\pi}{4}+\dfrac{t}{2}\Bigr)
					+{\it MathieuC} \Bigl(0,-\dfrac{2}{k(t)},\dfrac{-\pi}{4}+\dfrac{t}{2}\Bigr)\\
					{\it MathieuSPrime} \Bigl(0,-\dfrac{2}{k(t)},\dfrac{-\pi}{4}+\dfrac{t}{2}\Bigr)\Bigr)
					+4c_{25}^{2}\\
					{\it MathieuS} \Bigl(0,-\dfrac{2}{k(t)},\dfrac{-\pi}{4}+\dfrac{t}{2}\Bigr)
					{\it MathieuSPrime} \Bigl(0,-\dfrac{2}{k(t)},\dfrac{-\pi}{4}+\dfrac{t}{2}\Bigr)\Bigr)\\
					\Biggm/\Bigl(c_{26}(1+8c_{16})^{2}
					\Bigl({\it MathieuC} \Bigl(0,-\dfrac{2}{k(t)},\dfrac{-\pi}{4}+\dfrac{t}{2}\Bigr)\Bigr)^{2}+4c_{25}c_{26}\\
					{\it MathieuC} \Bigl(0,-\dfrac{2}{k(t)},\dfrac{-\pi}{4}+\dfrac{t}{2}\Bigr)
					{\it MathieuS} \Bigl(0,-\dfrac{2}{k(t)},\dfrac{-\pi}{4}+\dfrac{t}{2}\Bigr)
					+4c_{25}^{2}\\
					\Bigl({\it MathieuS} \Bigl(0,-\dfrac{2}{k(t)},\dfrac{-\pi}{4}+\dfrac{t}{2}\Bigr)\Bigr)^{2}\Bigr),
					\end{multline*}
					\begin{multline*}
					q_{1}(t)=2c_{26}(1+8c_{16})^{2}{\it MathieuC} \Bigl(0,-\dfrac{2}{k(t)},\dfrac{-\pi}{4}+\dfrac{t}{2}\Bigr){\it MathieuCPrime}\\
					\Bigl(0,-\dfrac{2}{k(t)},\dfrac{-\pi}{4}+\dfrac{t}{2}\Bigr)\sin t
					+4c_{25}c_{26}\Bigl({\it MathieuCPrime}\\
					\Bigl(0,-\dfrac{2}{k(t)},\dfrac{-\pi}{4}+\dfrac{t}{2}\Bigr)
					{\it MathieuS} \Bigl(0,-\dfrac{2}{k(t)},\dfrac{-\pi}{4}+\dfrac{t}{2}\Bigr)
					+{\it MathieuC} \\
					\Bigl(0,-\dfrac{2}{k(t)},\dfrac{-\pi}{4}+\dfrac{t}{2}\Bigr)
					{\it MathieuSPrime} \Bigl(0,-\dfrac{2}{k(t)},\dfrac{-\pi}{4}+\dfrac{t}{2}\Bigr)\Bigr)\sin t
					+8c_{25}^{2}\\
					{\it MathieuS}\Bigl(0,-\dfrac{2}{k(t)},\dfrac{-\pi}{4}+\dfrac{t}{2}\Bigr){\it MathieuSPrime} \Bigl(0,-\dfrac{2}{k(t)},\dfrac{-\pi}{4}+\dfrac{t}{2}\Bigr)\\
					\sin t+c_{26}(1+8c_{16})^{2}\Bigl({\it MathieuC} \Bigl(0,-\dfrac{2}{k(t)},\dfrac{-\pi}{4}+\dfrac{t}{2}\Bigr)\Bigr)^{2}\cos t+4c_{25}c_{26}\\
					{\it MathieuC}\Bigl(0,-\dfrac{2}{k(t)},\dfrac{-\pi}{4}+\dfrac{t}{2}\Bigr)
					{\it MathieuS} \Bigl(0,-\dfrac{2}{k(t)},\dfrac{-\pi}{4}+\dfrac{t}{2}\Bigr)\cos t\\
					+4c_{25}^{2}\Bigl({\it MathieuS} \Bigl(0,-\dfrac{2}{k(t)},\dfrac{-\pi}{4}+\dfrac{t}{2}\Bigr)\Bigr)^{2}\cos t,
					\end{multline*}
					and,\\
					\begin{multline*}
					q_{2}(t)=k(t)\Bigl(c_{26}(1+8c_{16})^{2}
					\Bigl({\it MathieuC} \Bigl(0,-\dfrac{2}{k(t)},\dfrac{-\pi}{4}+\dfrac{t}{2}\Bigr)\Bigr)^{2}\\
					+4c_{25}c_{26}
					{\it MathieuC} \Bigl(0,-\dfrac{2}{k(t)},\dfrac{-\pi}{4}+\dfrac{t}{2}\Bigr)
					{\it MathieuS} \\
					\Bigl(0,-\dfrac{2}{k(t)},\dfrac{-\pi}{4}+\dfrac{t}{2}\Bigr)
					+4c_{25}^{2}\Bigl({\it MathieuS} \Bigl(0,-\dfrac{2}{k(t)},\dfrac{-\pi}{4}+\dfrac{t}{2}\Bigr)\Bigr)^{2}\Bigr).
					\end{multline*}
				\end{minipage}& 
				\begin{minipage}{5cm}
					\begin{multline*}
					\zeta_{1}^{*}=\Bigl({\it MathieuC} \Bigl(0,-\dfrac{2}{k(t)},\\
					\dfrac{-\pi}{4}+\dfrac{t}{2}\Bigr)\Bigr)^{2}\dfrac{\partial}{\partial t}
					\\+\dfrac{x}{2}{\it MathieuC} \Bigl(0,-\dfrac{2}{k(t)},\dfrac{-\pi}{4}+\dfrac{t}{2}\Bigr)\\ {\it MathieuCPrime} \Bigl(0,-\dfrac{2}{k(t)},\\\dfrac{-\pi}{4}+\dfrac{t}{2}\Bigr)\dfrac{\partial}{\partial x}
					\end{multline*}
					\begin{multline*}
					\zeta_{2}^{*}=\Bigl( {\it MathieuS} \Bigl(0,-\dfrac{2}{k(t)},\\\dfrac{-\pi}{4}
					+\dfrac{t}{2}\Bigr)  \Bigr) ^{2}\dfrac{\partial}{\partial t}\\
					+\dfrac{x}{2}{\it MathieuS}\Bigl(0,-\dfrac{2}{k(t)},\dfrac{-\pi}{4}+\dfrac{t}{2}\Bigr) \\
					{\it MathieuSPrime} \Bigl(0,-\dfrac{2}{k(t)},\\\dfrac{-\pi}{4}+\dfrac{t}{2}\Bigr)\dfrac{\partial}{\partial x}
					\end{multline*}
					\begin{multline*}
					\zeta_{3}^{*}={\it MathieuC} \Bigl(0,-\dfrac{2}{k(t)},\\\dfrac{-\pi}{4}+\dfrac{t}{2}\Bigr) {\it MathieuS} \Bigl(0,-\dfrac{2}{k(t)},\\\dfrac{-\pi}{4}+\dfrac{t}{2}\Bigr)\dfrac{\partial}{\partial t}\\
					+\dfrac{x}{4}\Bigl({\it MathieuCPrime} \Bigl(0,-\dfrac{2}{k(t)},\\\dfrac{-\pi}{4}+\dfrac{t}{2}\Bigr) {\it MathieuS} \Bigl(0,-\dfrac{2}{k(t)},\\\dfrac{-\pi}{4}+\dfrac{t}{2}\Bigr)
					+{\it MathieuC} \Bigl(0,-\dfrac{2}{k(t)},\\\dfrac{-\pi}{4}+\dfrac{t}{2}\Bigr) {\it MathieuSPrime} \Bigl(0,\\-\dfrac{2}{k(t)},\dfrac{-\pi}{4}+\dfrac{t}{2}\Bigr)\Bigr)\dfrac{\partial}{\partial x}
					\end{multline*}
					$	\zeta_{4}^{*}=\dfrac{x}{2}\dfrac{\partial}{\partial x},$\\
					$ \zeta_{5}^{*}=\rho(t)\dfrac{\partial}{\partial x}. $
				\end{minipage}\\

				\hline
				
		\end{tabular}}
	\end{adjustbox}
\end{table}

\begin{table}
	\caption{Group Classification of the Second Order Neutral Differential Equation}
	\label{tab:4}       
	\begin{adjustbox}{width=1\textwidth}
		{\small
			\begin{tabular}{|l|r|}
				\hline 
				Type of Second order Neutral Differential Equation & Generators \\ 
				\hline 
				\begin{minipage}{11cm}
					$x''(t)+c(t)x(t)+d(t)x(t-r)+k(t)x''(t-r)=0,\\$ $d(t)=t^{m},$ where $m$ is a constant,\\
					$$c(t)=\left[\int_{}^{}-\dfrac{l(t)}{c_{2}j(t)}e^{-4\bigint\dfrac{e(t)}{j(t)}\mathrm{d}t}\mathrm{d}t+c_{30}\right]e^{4\bigint\dfrac{y(t)}{j(t)}}\mathrm{d}t.$$\\
					With $\theta=\Bigl(2\,{\frac {\sqrt {(k(t))^{-1}}{t}^{m/2+1}}{m+2}}\Bigr)$,
					\begin{multline*}
					e(t)=\sqrt{(k(t))^{-1}}(1+2c_{16})c_{29}^{2}{\sl J}_{(3+m)(m+2)^{-1}}(\theta){\sl J}_{(m+2)^{-1}}(\theta)t^{m/2+1}\\
					+2\sqrt{(k(t))^{-1}}c_{28}c_{29}\Bigl({\sl J}_{(3+m)(m+2)^{-1}}(\theta){\sl Y}_{(m+2)^{-1}}(\theta)\\
					+{\sl J}_{(m+2)^{-1}}(\theta){\sl Y}_{(3+m)(m+2)^{-1}}(\theta)\Bigr)t^{m/2+1}\\
					+4\sqrt{(k(t))^{-1}}c_{28}^{2}{\sl Y}_{(3+m)(m+2)^{-1}}(\theta){\sl Y}_{(m+2)^{-1}}(\theta)t^{m/2+1}\\
					-(1+2c_{16})c_{29}^{2}\Bigl({\sl J}_{(m+2)^{-1}}(\theta)\Bigr)^{2}-4c_{28}{\sl Y}_{(m+2)^{-1}}(\theta)\\
					\Bigl({\sl J}_{(m+2)^{-1}}(\theta)c_{29}+{\sl Y}_{(m+2)^{-1}}(\theta)c_{28}\Bigr),
					\end{multline*}
					\begin{multline*}
					j(t)=t\Bigl[(1+2c_{16})c_{29}^{2}\Bigl({\sl J}_{(m+2)^{-1}}(\theta)\Bigr)^{2}+4c_{28}{\sl Y}_{(m+2)^{-1}}(\theta)\\
					\Bigl({\sl J}_{(m+2)^{-1}}(\theta)c_{29}+{\sl Y}_{(m+2)^{-1}}(\theta)c_{28}\Bigr)\Bigr],
					\end{multline*}
					\begin{multline*}
					l(t)=4\sqrt{(k(t))^{-1}}(1+2c_{16})c_{29}^{2}{\sl J}_{(3+m)(m+2)^{-1}}(\theta){\sl J}_{(m+2)^{-1}}(\theta)t^{3m/2+1}\\
					-2mc_{29}^{2}c_{16}\Bigl({\sl J}_{(m+2)^{-1}}(\theta)\Bigr)^{2}t^{m}+8\sqrt{(k(t))^{-1}}c_{28}c_{29}\Bigl({\sl J}_{(3+m)(m+2)^{-1}}(\theta)\\
					{\sl Y}_{(m+2)^{-1}}(\theta)
					+{\sl J}_{(m+2)^{-1}}(\theta){\sl Y}_{(3+m)(m+2)^{-1}}(\theta)\Bigr)t^{3m/2+1}\\
					+16\sqrt{(k(t))^{-1}}c_{28}^{2}{\sl Y}_{(3+m)(m+2)^{-1}}(\theta){\sl Y}_{(m+2)^{-1}}(\theta)t^{3m/2+1}\\
					-8c_{29}^{2}c_{16}\Bigl({\sl J}_{(m+2)^{-1}}(\theta)\Bigr)^{2}t^{m}-mc_{29}^{2}\Bigl({\sl J}_{(m+2)^{-1}}(\theta)\Bigr)^{2}t^{m}\\
					-4mc_{28}c_{29}{\sl J}_{(m+2)^{-1}}(\theta){\sl Y}_{(m+2)^{-1}}(\theta)t^{m}\\
					-4mc_{28}^{2}\Bigl({\sl Y}_{(m+2)^{-1}}(\theta)\Bigr)^{2}t^{m}-4c_{29}^{2}\Bigl({\sl J}_{(m+2)^{-1}}(\theta)\Bigr)^{2}t^{m}\\
					-16c_{28}{\sl Y}_{(m+2)^{-1}}(\theta)
					\Bigl({\sl J}_{(m+2)^{-1}}(\theta)c_{29}+{\sl Y}_{(m+2)^{-1}}(\theta)c_{28}\Bigr)t^{m},
					\end{multline*}
					\begin{multline*}
					y(t)=\sqrt{(k(t))^{-1}}c_{29}^{2}(1+2c_{16}){\sl J}_{(3+m)(m+2)^{-1}}(\theta){\sl J}_{(m+2)^{-1}}(\theta)t^{m/2+1}\\
					+2\sqrt{(k(t))^{-1}}c_{28}c_{29}\Bigl({\sl J}_{(3+m)(m+2)^{-1}}(\theta){\sl Y}_{(m+2)^{-1}}(\theta)+{\sl J}_{(m+2)^{-1}}(\theta)\\
					{\sl Y}_{(3+m)(m+2)^{-1}}(\theta)\Bigr)t^{m/2+1}+4\sqrt{(k(t))^{-1}}c_{28}^{2}{\sl Y}_{(3+m)(m+2)^{-1}}(\theta)\\
					{\sl Y}_{(m+2)^{-1}}(\theta)t^{m/2+1}-(1+2c_{16})c_{29}^{2}\Bigl({\sl J}_{(m+2)^{-1}}(\theta)\Bigr)^{2}\\
					-4c_{28}{\sl Y}_{(m+2)^{-1}}(\theta)
					\Bigl({\sl J}_{(m+2)^{-1}}(\theta)c_{29}+{\sl Y}_{(m+2)^{-1}}(\theta)c_{28}\Bigr).
					\end{multline*}
				\end{minipage}& 
				\begin{minipage}{5cm}
					With $\nu=(m+2)^{-1},\\ \tau=\sqrt{(k(t))^{-1}}t^{m/2+1}, \mu=2\tau\nu,$\\
					\begin{multline*}
					\zeta_{1}^{*}=t(J_{\nu}(\mu))^{2}\dfrac{\partial}{\partial t}\\
					+x\Bigl(\dfrac{1}{2}(J_{\nu}(\mu))^{2}+\dfrac{2}{m+2}J_{\nu}(\mu)\\
					\Bigl(-J_{\nu+1}(\mu)+\dfrac{J_{\nu}(\mu)}{2\tau}\Bigr)\\\tau(m/2+1)\Bigr)\dfrac{\partial}{\partial x}
					\end{multline*}
					\begin{multline*}
					\zeta_{2}^{*}=t(Y_{\nu}(\mu))^{2}\dfrac{\partial}{\partial t}\\
					+x\Bigl(\dfrac{1}{2}(Y_{\nu}(\mu))^{2}+\dfrac{2}{m+2}Y_{\nu}(\mu)\\
					\Bigl(-Y_{\nu+1}(\mu)+\dfrac{Y_{\nu}(\mu)}{2\tau}\Bigr)\\\tau(m/2+1)\Bigr)\dfrac{\partial}{\partial x}
					\end{multline*}
					\begin{multline*}
					\zeta_{3}^{*}=tJ_{\nu}(\mu)Y_{\nu}(\mu)\dfrac{\partial}{\partial t}\\
					+x\Bigl(\dfrac{1}{2}J_{\nu}(\mu)Y_{\nu}(\mu)\\
					+\dfrac{1}{m+2}\Bigl(\Bigl(-J_{\nu+1}(\mu)\\
					+\dfrac{J_{\nu}(\mu)}{2\tau}Y_{\nu}(\mu)+J_{\nu}(\mu)\Bigl(-Y_{\nu+1}(\mu)\\
					+\dfrac{Y_{\nu}(\mu)}{2\tau}\Bigr)\Bigr)\tau(m/2+1)\Bigr)\Bigr)\dfrac{\partial}{\partial x}
					\end{multline*}
					$	\zeta_{4}^{*}=\dfrac{x}{2}\dfrac{\partial}{\partial x},$\\
					$ \zeta_{5}^{*}=\rho(t)\dfrac{\partial}{\partial x}. $
				\end{minipage}\\

				\hline
		\end{tabular}}
	\end{adjustbox}
\end{table}

\begin{table}
	\caption{Group Classification of the Second Order Delay Differential Equation}
	\label{tab:5}      
	\begin{tabular}{|c|c|}
		\hline 
		Type of Second order Delay Differential Equation & Generators \\ 
		\hline 
		\begin{minipage}{9cm}
			$x''(t)+b(t)x'(t-r)+c(t)x(t)+d(t)x(t-r)=0,$ \\ $d(t)=c_{32}b^{2}(t)+\dfrac{b'(t)}{2},$ \\ $c(t)=\dfrac{1}{2}\left[c_{33}b^{2}(t)-\dfrac{3}{2}\left(\dfrac{b'(t)}{b(t)}\right)^{2}+\dfrac{b''(t)}{b(t)}\right] $ 
		\end{minipage}&\begin{minipage}{6cm}
			$ \zeta_{1}^{*}=x\dfrac{\partial}{\partial x},\\\zeta_{2}^{*}=\dfrac{1}{b(t)}\dfrac{\partial}{\partial t}+\dfrac{x}{2}\left(\dfrac{1}{b(t)}\right)'\dfrac{\partial}{\partial x},\\ \zeta_{3}^{*}=\rho(t)\dfrac{\partial}{\partial x}$ 
		\end{minipage} \\ 
		\hline 
		\begin{minipage}{7cm}
			$x''(t)+b(t)x'(t-r)+c(t)x(t)=0,$ \\$c(t)=\dfrac{c_{36}}{c_{35}^{2}}$
		\end{minipage}& \begin{minipage}{6cm}
			$\zeta_{1}^{*}=\dfrac{\partial}{\partial t},$ \\$\zeta_{2}^{*}=x\dfrac{\partial}{\partial x} $,\\$ \zeta_{3}^{*}=\rho(t)\dfrac{\partial}{\partial x}.$ 
		\end{minipage}\\ 
		\hline 
		\begin{minipage}{9cm}
			$x''(t)+c(t)x(t)+d(t)x(t-r)=0,$\\ $c(t)=\dfrac{1}{2}\left[\dfrac{c_{31}}{c_{37}}d(t)+\dfrac{d''(t)}{2d(t)}-\dfrac{5}{8}\left(\dfrac{d'(t)}{d(t)}\right)^{2}\right]$ 
		\end{minipage}& \begin{minipage}{6cm}
			$\zeta_{1}^{*}=\dfrac{1}{\sqrt{d(t)}}\dfrac{\partial}{\partial t},$\\ $\zeta_{2}^{*}=\left[\left(-\dfrac{d'(t)}{d^{3/2}(t)}\right)x\right]\dfrac{\partial}{\partial x} $,\\$\zeta_{3}^{*}=\dfrac{x}{2}\dfrac{\partial}{\partial x},$ \\$ \zeta_{4}^{*}=\rho(t)\dfrac{\partial}{\partial x}.$
		\end{minipage}\\ 
		\hline 
	\end{tabular}
\end{table}

\bibliographystyle{unsrt}  


\end{document}